\documentclass[12pt]{article}
\title{Definable groups and fields in t-minimal theories}
\author{Will Johnson}

\usepackage{amsmath, amssymb, amsthm}    	% need for subequations
\usepackage{fullpage} 	% without this, will have wide math-article-paper margins
\usepackage{amscd}
\usepackage{hyperref}
\usepackage[all]{xy}
\usepackage{centernot}
\usepackage{enumitem}

\DeclareMathOperator*{\forkindep}{\raise0.2ex\hbox{\ooalign{\hidewidth$\vert$\hidewidth\cr\raise-0.9ex\hbox{$\smile$}}}}

\newcommand{\rra}{\rightrightarrows}

\newcommand{\alg}{\mathrm{alg}}
\newcommand{\sep}{\mathrm{sep}}

\newcommand{\ACF}{\operatorname{ACF}}

\newcommand{\mult}{\operatorname{mult}}

\newcommand{\res}{\operatorname{res}}

\newcommand{\Aut}{\operatorname{Aut}}

\newcommand{\id}{\operatorname{id}}

\newcommand{\acl}{\operatorname{acl}}
\newcommand{\im}{\operatorname{im}}
\newcommand{\dcl}{\operatorname{dcl}}
\newcommand{\tp}{\operatorname{tp}}

\newcommand{\val}{\operatorname{val}}

\newcommand{\trdeg}{\operatorname{tr.deg}}
\newcommand{\bd}{\operatorname{bd}}

\newcommand{\cl}{\operatorname{cl}}

\newcommand{\ter}{\operatorname{int}}

\newtheorem{theorem}{Theorem}[section] % numbered like the section
\newtheorem{nontheorem}[theorem]{Non-Theorem}
\newtheorem{lemma}[theorem]{Lemma}

\newtheorem{corollary}[theorem]{Corollary}
\newtheorem{fact}[theorem]{Fact}

\newtheorem{assumption}[theorem]{Assumption}

\newtheorem{question}[theorem]{Question}
\newtheorem{proposition}[theorem]{Proposition}
\newtheorem*{theorem-star}{Theorem}
\newtheorem*{lemma-star}{Theorem}
\newtheorem*{conjecture-star}{Conjecture}

\theoremstyle{definition}
\newtheorem{definition}[theorem]{Definition}
\newtheorem{example}[theorem]{Example}

\newtheorem{remark}[theorem]{Remark}

\newtheorem{claim}[theorem]{Claim}

\newtheorem*{acknowledgment}{Acknowledgments}

\theoremstyle{remark}
\newtheorem{nonclaim}[theorem]{Non-Claim}

\newcommand{\Qq}{\mathbb{Q}}

\newcommand{\Nn}{\mathbb{N}}

\newcommand{\Mm}{\mathbb{M}}

\newcommand{\Pp}{\mathbb{P}}
\newcommand{\Ll}{\mathcal{L}}

\newcommand{\ba}{{\bar{a}}}
\newcommand{\bb}{{\bar{b}}}

\newcommand{\bx}{{\bar{x}}}
\newcommand{\by}{{\bar{y}}}

\newenvironment{claimproof}[1][\proofname]
               {
                 \proof[#1]
                 
               }
               {
                 \endproof
               }

\begin{document}
\maketitle

\begin{abstract}
  Let $T$ be a theory which is t-minimal, meaning that with respect to
  some definable topology, a unary definable set $D \subseteq M$ has
  non-empty interior iff it is infinite.  If $K$ is a definable field
  in $T$, then $K$ is finite or ``large'' in the sense of Pop: any
  smooth algebraic curve $C$ over $K$ with at least one $K$-rational
  point has infinitely many $K$-rational points.  We also assign a
  canonical topology to any abelian definable group $G$ in a t-minimal
  theory.  In the case where the t-minimal theory is ``visceral'' in
  the sense of Dolich and Goodrick, meaning that the definable
  topology is induced by a definable uniformity, we can drop the
  assumption of abelianity of $G$, and the resulting topology on $G$
  is a definable manifold in the style of Acosta L\'opez and Hasson.
\end{abstract}

\section{Introduction}

By a ``definable family'' we mean a family $\{D_a\}_{a \in X}$ where
$D_a = \{b : (a,b) \in D\}$ for some definable sets $D, X$.  A
``definable topology'' on a definable set $D$ is a topology $\tau$
such that some definable family is a basis of opens for $\tau$.  A
1-sorted theory $T$ is \emph{topologically minimal} or
\emph{t-minimal} in the sense of
Mathews~\cite[Definition~2.4]{mathews} if there's a definable
Hausdorff topology on models of $T$ (on the home sort $M^1$), such
that for any unary definable set $D \subseteq M$, we have
\begin{equation*}
  \ter(D) \ne \varnothing \iff |D| = \infty,
\end{equation*}
where $\ter(D)$ denotes the interior of $D$.  An equivalent condition
is that
\begin{itemize}
\item There are no isolated points in $M$, and
\item Every unary definable set $D \subseteq M$ has finite boundary.
\end{itemize}
Many important classes of theories are t-minimal, such as (dense)
o-minimal theories, P-minimal theories, (dense) C-minimal theories,
and hensel minimal theories.  Many of these examples satisfy a
stronger condition---they are \emph{visceral} in the sense of Dolich
and Goodrick~\cite[Definition~3.3]{visceral}.  The operational
difference between t-minimal and visceral is that in a visceral
theory, the definable topology is induced by a definable
\emph{uniformity} (i.e., a ``uniform structure'' in the sense of
pointset topology).  See~\cite{visceral} for a precise formulation.
This change makes a world of difference: visceral theories have
theorems of generic continuity and cell decomposition
\cite{simonWalsberg,visceral,wj-visc-1}.  Oddly, these properties
can fail in a t-minimal theory \cite[Remark~1.16]{wj-visc-1}.
Nevertheless, t-minimality yields a peculiar sort of dimension theory
on definable sets \cite[\S2]{wj-visc-1}, whose principal defect is
that a definable surjection $f : X \to Y$ can have $\dim(X) <
\dim(Y)$.  (On the other hand, dimension behaves as expected for
definable injections, cartesian products, and unions, and it's even
definable and subadditive in some sense.)

The point of this paper is that from even the relatively weak
condition of t-minimality, one can say something non-trivial about
definable (not interpretable!) fields.
\subsection{Definable fields}
Recall that a field
$(K,+,\cdot)$ is \emph{large} if it satisfies one of the following
equivalent conditions~\cite{Pop-little}:
\begin{itemize}
\item In explicit terms, if $P(x,y)$ is a polynomial over $K$ and
  $P(a,b) = 0 \ne \frac{\partial P}{\partial y}(a,b)$ for some $(a,b)
  \in K^2$, then the zeroset $\{(x,y) \in K^2 : P(x,y)=0\}$ is
  infinite.
\item In algebro-geometric terms, if $C$ is a smooth curve over $K$
  and the set of $K$-rational points $C(K)$ is non-empty, then $C(K)$
  is infinite.
\end{itemize}
For example, $\Qq$ fails to be large because the equation $x^4 + y^4 -
1 = 0$ has only four solutions by the $n=4$ case of Fermat's last
theorem.  On the other hand, most fields with a tame model theory are
large (or finite).  For example, until the recent examples of
``curve-excluding fields'' were discovered \cite{CXF}, one could say
that all known examples of fields with decidable first-order theories
were large or finite.

We can now state our main theorem.
\begin{theorem}[{=Corollary~\ref{target}}] \label{main-1}
  If $K$ is a definable field in a t-minimal theory, then $K$ is large
  or finite.
\end{theorem}
For example, the field $\Qq$ is not definable in any t-minimal theory,
nor are the non-large curve-excluding fields of \cite{CXF}.
\begin{remark}
  It is essential that we say ``definable'' rather than
  ``interpretable'' in the theorem.  For example, the field $\Qq((t))$
  of Laurent series over $\Qq$ is known to be t-minimal, but it
  interprets the non-large field $\Qq$.
\end{remark}
\begin{remark}
  The non-definability of $\Qq$ in t-minimal theories can be seen more
  directly from the classic work of J. Robinson showing that all
  computable functions are definable in $\Qq$.  In particular, there
  is a definable bijection $f : \Qq^2 \to \Qq$.  If $K$ is definable
  in a t-minimal theory and $K \equiv \Qq$, then $K$ would define a
  bijection $K^2 \to K$.  However, this contradicts the properties of
  dimension in t-minimal theories: the bijection would show $\dim(K) =
  \dim(K^2) = 2 \dim(K)$, so $\dim(K)=0$ and $K$ is finite.  (A
  similar argument shows that definable fields must be
  perfect~\cite[Corollary~2.53]{wj-visc-1}.)

  On the other hand, such arguments cannot be used to show that
  t-minimal theories fail to define the curve-excluding fields of
  \cite{CXF}, as CXF is a geometric theory with an excellent notion of
  dimension for definable sets \cite[Theorem~1.11(2)]{CXF}.  So we
  \emph{are} proving something new.
\end{remark}
The proof of Theorem~\ref{main-1} involves building a canonical
topology on $K$:
\begin{theorem}[{$\subseteq$ Theorem~\ref{fieldtop} $\cup$ Corollary~\ref{target}}] \label{main-2}
  Let $K$ be an infinite definable field in a t-minimal theory.  There
  is a unique definable field topology $\tau$ on $K$ such that the
  following condition holds:
  \begin{itemize}
  \item $\ter(D) \ne \varnothing \iff \dim(D) = \dim(K)$ for definable
    $D \subseteq K$.
  \end{itemize}
  Moreover, $\tau$ satisfies the following conditions:
  \begin{itemize}
  \item $\tau$ is Hausdorff and non-discrete.
  \item Let $P(x) \in K[x]$ be a separable polynomial of degree $d$.
    If $a \in K$ is a simple root of $P$ and $U \ni a$ is a
    $\tau$-neighborhood, then the set
    \begin{equation*}
      \{Q(x) \in K[x] : Q \text{ has a root in } U\}
    \end{equation*}
    is a neighborhood of $P$.
  \end{itemize}
\end{theorem}
We say a field topology $\tau$ is \emph{st-henselian} if it satisfies
final condition of Theorem~\ref{main-2}.  If one drops the word
``separable'', this turns into the \emph{generalized t-henselianity}
or \emph{gt-henselianity} of Dittmann, Walsberg, and Ye
\cite{hensquot2}.  It would be nice to strengthen Theorem~\ref{main-2}
from st-henselianity to gt-henselianity, but I could not see how to do
it.  The connection to largeness (Theorem~\ref{main-1}) is that
st-henselian fields are large (Theorem~\ref{bronze-large}).  In fact,
one only needs a weaker condition called
\emph{bt-henselianity}.\footnote{The initials g, s, and b stand for
gold, silver, and bronze, or if you like ``generalized
t-henselianity'', ``shoddy t-henselianity'', and ``barely
t-henselianity''.}  See \S\ref{sec-vars} for more about these
conditions and their connection to largeness.
\begin{question}\label{psf-yeesh}
  Can a pseudofinite field be definable in a t-minimal theory?
\end{question}
Pseudofinite fields are large, so Theorem~\ref{main-1} does not
exclude this possibility.  My original hope was to answer
Question~\ref{psf-yeesh} negatively by proving the following:
\begin{itemize}
\item The canonical topology on a definable field in a t-minimal
  theory is gt-henselian.
\item Pseudofinite fields do not admit gt-henselian topologies.\footnote{Gt-henselian topologies generalize the ``t-henselian'' topologies of Prestel and Ziegler \cite{prestel-ziegler}, and pseudofinite fields \emph{definitely} don't admit t-henselian topologies.}
\end{itemize}
I was unable to prove the first point, and the second point turns out
to be false---any countable pseudofinite field admits a gt-henselian
topology \cite{large-gt}.  On the other hand, it feels rather
outlandish that one could put a reasonable st-henselian topology on a
pseudofinite field and make definable sets be topologically
tame.\footnote{For example, in characteristic $\ne 2$, one can use
st-henselianity to show that the set of non-zero squares is open.
Given that this set is ``random'' in some sense, one would have
expected instead for the set of squares to be dense and codense.}
Thus Question~\ref{psf-yeesh} remains wide open.

Returning to the general setting of an infinite definable field $K$ in
a t-minimal theory, suppose that $\dim(K)=1$.  Then the first point of
Theorem~\ref{main-2} shows that the induced structure on $K$ is
t-minimal---even visceral---with respect to the canonical topology
$\tau$.\footnote{Definable field topologies yield definable
uniformities by translating neighborhoods of 0.}  It follows that all
the topological tameness theorems of \cite{visceral,wj-visc-1} hold,
such as generic continuity, cell decomposition, etc., \emph{even if
these properties failed in the base t-minimal theory}.

Focusing in further, suppose $T$ is a theory of fields, expanded by
extra structure, which is t-minimal with respect to some topology
$\tau_0$.  Then Theorem~\ref{main-2} gives a new \underline{field}
topology $\tau$ making $T$ into a \underline{visceral} theory.  One
should think of $\tau$ as $\tau_0$ with the problems fixed.  A typical
example is the case where $T$ is RCF and $\tau_0$ is the Sorgenfrey
topology---the one with basic open sets $[a,b)$.  This topology isn't
  a field topology, and there are definable functions like $f(x)=-x$
  that are nowhere continuous.  The topology $\tau$ is the correct,
  standard topology.

In particular, t-minimal theories of fields are automatically
visceral, without assuming that the original topology respected the
field operations.

If $K$ is a definable field with $\dim(K)>1$, then $K$ with the
induced structure and canonical topology won't be visceral.
Nevertheless, one can prove some tame topology theorems for $K$, such
as the generic continuity of definable functions $f : U \to K^m$ with
$U \subseteq K^n$ \underline{open} Theorem~\ref{gencon}.  It's less
clear what happens for $f : D \to K^m$ with $D \subseteq K^n$
arbitrary.
\begin{question}
  Which further ``tame topology'' theorems can we proven in
  $(K,\tau)$, where $K$ is a definable infinite field in a t-minimal
  theory and $\tau$ is its canonical topology from
  Theorem~\ref{main-2}?
\end{question}

\subsection{Definable groups}
Underlying the theorems on definable fields are some theorems on
definable groups.
\begin{theorem}[{= Theorem~\ref{cantop} $\cup$ Lemma~\ref{characterization} $\cup$ Theorem~\ref{non-triv}}] \label{main-3}
  Let $(G,+)$ be a definable abelian group in a t-minimal theory.
  Then there is a unique definable group topology $\tau$ on $G$ with
  the following property:
  \begin{itemize}
  \item $\ter(D) \ne \varnothing \iff \dim(D) = \dim(G)$ for definable
    $D \subseteq G$.
  \end{itemize}
  Moreover, $\tau$ is non-trivial when $|G|>1$.
\end{theorem}
Unfortunately, I was not able to drop the assumption of abelianity,
nor to prove that $\tau$ is Hausdorff, except in the case of fields.
In the setting of Theorem~\ref{main-3}, one can prove some weak
results about tame topology (see \S\ref{apof}, especially
Theorems~\ref{gencon} and \ref{gencon2}).  As a corollary,
homomorphisms are continuous (Corollary~\ref{homcor}).

The situation is much better in visceral theories, or more generally,
t-minimal theories with generic continuity of definable
correspondences.
\begin{theorem}[{$\approx$ Theorem~\ref{thm-unique}}] \label{main-4}
  Let $(G,\cdot)$ be a definable group in a visceral theory.  Then
  there is a unique definable group topology $\tau$ on $G$ making $G$
  into a definable manifold.  Moreover,
  \begin{itemize}
  \item $\tau$ is Hausdorff
  \item $\tau$ is non-discrete when $G$ is infinite.
  \item $\ter(D) \ne \varnothing \iff \dim(D) = \dim(G)$ for definable
    $D \subseteq G$.  In particular, $\tau$ agrees with the canonical
    topology from Theorem~\ref{main-3} when $G$ is abelian.
  \end{itemize}
\end{theorem}
See \S\ref{defman} for a precise definition of ``definable manifold'',
which is similar to the classic definition in the o-minimal setting
\cite{Peterzil-Steinhorn}, but allowing for finite covers \emph{a la}
Acosta L\'opez and Hasson~\cite{acosta-hasson}.  The tame topology for
definable manifolds is much better than in the vague setting of
Theorem~\ref{main-3}; see \S\ref{tame-man}.

%% In a t-minimal theory of fields with \emph{generic differentiability},
%% I suppose one could define a more precise notion of ``definable
%% $C^k$-manifold'' and show that definable groups have an adjoint
%% action.  We leave this for further work---except for the case of
%% definable fields, which seems like a shame to omit.

%% Further directions: generic differentiability etc.

\subsection{Outline}

The paper divides into three parts.
\begin{itemize}
\item In Sections~\ref{sec-vars}--\ref{via-dims}, we focus on the
  abstract conditions of gold/silver/bronze-t-henselianity.  We show
  that these conditions imply largeness (Theorem~\ref{bronze-large}).
  We also give an abstract dimension-theoretic condition which ensures
  that these topological conditions hold (Theorems~\ref{A-thm},
  \ref{silver-theorem}).  This may be of independent interest.
  Nothing in these sections is specific to t-minimal theories, and I
  expect the criterion of Theorems~\ref{A-thm} and
  \ref{silver-theorem} to have further applications in later papers.
\item In Sections~\ref{new-insert}--\ref{atdf}, we focus on the case of
  visceral theories.  We build up the machinery of definable manifolds
  and their tame topology, and use the usual methods to prove
  Theorems~\ref{main-4} and \ref{main-2}.
\item In Sections~\ref{ground}--\ref{elevensy} we focus on the much
  harder case of t-minimal theories, proving Theorems~\ref{main-3} and \ref{main-2} in this case.
\end{itemize}
The reason for handling the visceral case separately in
Sections~\ref{new-insert}--\ref{atdf} is because the proofs are easier and
the results are stronger.

Appendix~\ref{not-nice} explains why we need finite covers in our
definition of ``definable manifold'', and Appendix~\ref{app-B} gives
the lengthy proof of a technical lemma needed to deal with the
t-minimal case in Sections~\ref{ground}--\ref{elevensy}.

% PART 1: largeness and field topologies and the abstract criterion

\section{Variants of gt-henselianity} \label{sec-vars}
\begin{definition} \label{def:gold}
  Let $K$ be a field.  A field topology $\tau$ on $K$ is
  \emph{generalized t-henselian} or \emph{gold t-henselian} if the
  following equivalent conditions hold:
  \begin{enumerate}
  \item \label{v1} For any $d \ge 2$ and neighborhood $U \ni -1$,
    there is a neighborhood $V \ni 0$ such that if $a_0, a_1, \ldots,
    a_{d-2} \in V$, then the polynomial $X^d + X^{d-1} +
    a_{d-2}X^{d-2} + \cdots + a_1X + a_0$ has a root in $U$.
  \item \label{v2} If $f : K^n \to K^n$ is a polynomial map and the Jacobian
    matrix of $f$ at some point $\ba \in K^n$ is invertible, then
    $f$ is a local homeomorphism at $\ba$, meaning that there are open
    neighborhoods $U \ni \ba$ and $V \ni f(\ba)$ such that $f$ induces
    a homeomorphism $U \to V$.
  \item \label{v3} If $f : K^n \times K^m \to K^m$ is a polynomial map such that
    $f(\ba,\bb) = \bar{0}$ and the matrix $\frac{\partial}{\partial
    y} f(\bx,\by)$ is invertible at $(\ba,\bb)$, then there are
    neighborhoods $U \ni \ba$ and $V \ni \bb$ and a continuous
    function $g : U \to V$ such that
    \begin{equation*}
      g(\bx) = \by \iff f(\bx,\by) = \bar{0} \text{ for } (\bx,\by)
      \in U \times V.
    \end{equation*}
  \item \label{v4} If $f : V \to W$ is an \'etale morphism of
    varieties over $K$, then $V(K) \to W(K)$ is a local homeomorphism,
    with respect to the topology induced by $\tau$.
  \end{enumerate}
\end{definition}
For the equivalence, see \cite[Proposition~6.2]{tops-rings}, though
the proof mostly comes from \cite{hensquot2}.  Condition (\ref{v1}) is
the original definition from \cite[Definition~8.1]{hensquot2}.
Conditions (\ref{v2}) and (\ref{v3}) are the inverse function theorem
and implicit function theorems for polynomials.  Condition (\ref{v4})
is a conceptually nice generalization of (\ref{v2}), which we will not
use.
\begin{fact}[{\cite[Proposition~8.3]{hensquot2}}] \label{isect}
  A topological field $(K,\tau)$ is t-henselian if it is gt-henselian
  and V-topological.
\end{fact}
See \cite[Sections~3, 7]{prestel-ziegler} for the definitions of
t-henselian topologies and V-topologies.
\begin{fact}[{\cite[Corollary~8.15]{hensquot2}}]
  If $(K,\tau)$ is a gt-henselian topological field, then $K$ is
  large.
\end{fact}
Our goal in this section is to define two weakenings of
gt-henselianity, and show that they too imply largeness.  First, we
reformulate gt-henselianity as a continuity of roots condition.
Following \cite[Section~7]{prestel-ziegler}, let $K[X]^d_1$ denote the
space of monic polynomials of degree $d$.  We topologize $K[X]^d_1$ by
identifying it with $K^d$.
\begin{proposition} \label{prop-gold}
  A topological field $(K,\tau)$ is gt-henselian if and only if the
  following condition holds:
  \begin{enumerate}
    \setcounter{enumi}{4}
  \item \label{v5} If $Q(X) \in K[X]^d_1$ has a simple root $a \in K$, then
    for any neighborhood $U$ of $a$, there is a neighborhood $Q \in V
    \subseteq K[X]^d_1$ such that if $P \in V$, then $P$ has a root in
    $U$.
  \end{enumerate}
\end{proposition}
\begin{proof}
  We compare Condition (\ref{v5}) to conditions (\ref{v1}) and
  (\ref{v3}) in Definition~\ref{def:gold}.
  \begin{description}
  \item[$(\ref{v5})\implies(\ref{v1}):$] Take $Q = X^d + X^{d-1}$ and
    $a = -1$.
  \item[$(\ref{v3})\implies(\ref{v5}):$] The proof is straightforward,
    but we include the details for completeness.  Write $Q(X)$ as $X^d
    + b_{d-1}X^{d-1} + \cdots + b_1X + b_0$.  Let
    $f(y_0,\ldots,y_{d-1},x) = x^d + y_{d-1}x^{d-1} + \cdots + y_1x +
    y_0$.  Then
    \begin{gather*}
      f(\bb,a) = Q(a) = 0 \\
      \frac{\partial}{\partial x}f(\bb,a) = Q'(a) \ne 0.
    \end{gather*}
    By the polynomial implicit function theorem, there are
    neighborhoods $U_0 \ni a$ and $V_0 \ni \bb$ and a continuous
    function $g : V_0 \to U_0$ such that
    \begin{equation*}
      x^d + y_{d-1}x^{d-1} + \cdots + y_1x + y_0 = 0 \iff x =
      g(y_0,\ldots,y_{d-1}) \text{ for } x \in U_0 \text{ and } \by
      \in V_0.
    \end{equation*}
    Let $V = g^{-1}(U \cap U_0)$.  Then for any $\by \in V$, the
    polynomial $X^d + y_{d-1}X^{d-1} + \cdots + y_1X + y_0$ has a root
    in $U$, namely $X = g(\by)$.  \qedhere
  \end{description}
\end{proof}
Our two weakenings of gt-henselianity are obtained by weakening the
criterion in Proposition~\ref{prop-gold}:
\begin{definition}
  Let $\tau$ be a field topology on $K$.
  \begin{enumerate}
  \item $\tau$ is \emph{shoddy t-henselian} or \emph{silver t-henselian} if
    the following holds: if $Q(X) \in K[X]^d_1$ is separable and has a
    root $a \in K$, then for any neighborhood $U$ of $a$, there is a
    neighborhood $Q \in V \subseteq K[X]^d_1$ such that if $P \in V$,
    then $P$ has a root in $U$.
  \item $\tau$ is \emph{barely t-henselian} or \emph{bronze t-henselian} if
    the following holds: if $Q(X) \in K[X]^d_1$ is separable and has a
    root $a \in K$, then there is a neighborhood $Q \in V \subseteq
    K[X]^d_1$ such that if $P \in V$, then $P$ has a root in $K$.
  \end{enumerate}
\end{definition}
Silver t-henselianity is weaker than gold t-henselianity, because we
only consider the case where $Q$ is separable---so \emph{all} the
roots of $Q$ over $K^{\alg}$ are simple, not just $a$.  Bronze
t-henselianity is weaker than silver t-henselianity, because we have
no control over where $P$'s roots are.
\begin{remark}
  Bronze t-henselianity can be rephrased more simply as follows: the
  set $\{P(X) \in K[X]^d_1 : P \text{ is separable and has a root in
    $K$}\}$ is open in $K[X]^d_1$.  To prove this, one needs only
  observe that the set of separable polynomials is open in $K[X]^d_1$
  (use discriminants).
\end{remark}
\noindent We care about st-henselianity because:
\begin{itemize}
\item Infinite definable fields in t-minimal theories will carry
  natural st-henselian topologies.
\item In a later paper, we will see that NIP large fields (other than separably closed fields) also carry natural st-henselian topologies.
\end{itemize}
We care about bt-henselianity because:
\begin{itemize}
\item Bronze t-henselianity is already sufficient to prove largeness.
\item Bronze t-henselianity gives a simpler proof that infinite
  definable fields in t-minimal theories are large.
\end{itemize}
\begin{question} \label{q-implies}
  Are the implications gt-henselian $\implies$ st-henselian $\implies$
  bt-henselian strict?
\end{question}
We will say more about this in a later paper.  For
now, we prove that bt-henselian fields are large.
\begin{lemma} \label{arc}
  Let $K/k$ be an extension of fields.  Let $P(x,y)$ be an irreducible
  polynomial over $k$, such that $P(0,0) = 0$, but the partial
  derivatives $\frac{\partial P}{\partial x}$ and $\frac{\partial
    P}{\partial y}$ do not both vanish at $(0,0)$.  Let $(a,b) \in
  K^2$ be such that $\{a,b\}$ is algebraically independent over $k$.
  Then the polynomial $P(z,az^2 + bz) \in K[z]$ is separable.
\end{lemma}
\begin{proof}
  Enlarging $K$, we may assume that $K$ is a monster model of $\ACF$.  Let
  \begin{gather*}
    P_1 = \frac{\partial P}{\partial x} \\ P_2 = \frac{\partial
      P}{\partial y}.
  \end{gather*}
  Note that
  \begin{equation*}
    \frac{d}{dz} P(z,az^2 + bz) = P_1(z,az^2+bz) + P_2(z,az^2+bz)(2az+b).
  \end{equation*}
  Suppose $P(z,az^2+bz) \in K[z]$ has a double root at $z = t$.  Then
  $P(z,az^2+bz)$ and its derivative both vanish at $z = t$, so
  \begin{gather*}
    P(t,at^2+bt) = 0 \\
    P_1(t,at^2+bt) + P_2(t,at^2+bt)(2at+b) = 0.
  \end{gather*}
  If $t = 0$, then the second line gives
  \begin{equation*}
    P_1(0,0) + P_2(0,0)b = 0.
  \end{equation*}
  Since $P_i(0,0) \in k$ for $i = 1, 2$ and $b$ is transcendental over
  $k$, this implies $P_1(0,0) = P_2(0,0) = 0$, contradicting the
  assumptions on $P$.

  Next suppose $t \ne 0$.  Let $s = at^2 + bt$, so that
  \begin{gather*}
    P(t,s) = 0 \\
    P_1(t,s) + P_2(t,s)(2at+b) = 0. \tag{$\ast$}
  \end{gather*}
  The equation $P(t,s) = 0$ shows that $\trdeg(s,t/k) \le 1$.  If $s,t
  \in k^{\alg}$, then the equation $s = at^2 + bt$ ensures $b \in
  k(a)^{\alg}$ (as $t \ne 0$), and this contradicts the algebraic
  independence of $a$ and $b$ over $k$.  So instead $\trdeg(s,t/k) =
  1$.  Then the ideal
  \begin{equation*}
    I = \{A(x,y) \in k[x,y] : A(s,t) = 0\}
  \end{equation*}
  is a height 1 prime ideal containing the irreducible polynomial
  $P(x,y)$, so it must be the principal ideal $P(x,y) \cdot k[x,y]$.
  Now break into cases:
  \begin{itemize}
  \item If $P_2(t,s) = 0$, then $P_1(t,s) = 0$ by ($\ast$), and both
    $P_1$ and $P_2$ are contained in $I$.  Then
    \begin{gather*}
      P_1(x,y) = Q_1(x,y)P(x,y) \\
      P_2(x,y) = Q_2(x,y)P(x,y)
    \end{gather*}
    for some polynomials $Q_1(x,y), Q_2(x,y) \in k[x,y]$.
    Substituting $(x,y) = (0,0)$, it follows that $P_1(0,0) =
    P_2(0,0)$, a contradiction.\footnote{The point is that $(s,t)$ is
    generic on the curve $C$ defined by $P$.  Since $P$ is irreducible
    and there is at least \emph{one} smooth point $(0,0)$, any generic
    point must be smooth.}
  \item If $P_2(t,s) \ne 0$, then
    \begin{align*}
      at^2 + bt &= s \\
      2at + b &= \frac{-P_1(t,s)}{P_2(t,s)}
    \end{align*}
    or equivalently,
    \begin{equation*}
      \begin{pmatrix} t^2 & t \\
        2t & 1
      \end{pmatrix} 
      \begin{pmatrix}
        a \\ b
      \end{pmatrix}
      =
      \begin{pmatrix}
        s \\ \frac{-P_1(t,s)}{P_2(t,s)}.
      \end{pmatrix}
    \end{equation*}
    The matrix on the left has determinant $t^2 - 2t^2 = -t^2 \ne 0$,
    so it is invertible.  Then $a,b \in k(s,t)$, contradicting the
    fact that $\trdeg(a,b/k) = 2$ and $\trdeg(s,t/k) = 1$.  \qedhere
  \end{itemize}
\end{proof}
\begin{theorem} \label{bronze-large}
  If $(K,\tau)$ is bt-henselian, then $K$ is large.
\end{theorem}
\begin{proof}
  Let $P(x,y)$ be a polynomial such that $P(0,0) = 0 \ne
  \frac{\partial}{\partial y}P(0,0)$.  We must show that $\{(a,b) \in
  K^2 : P(a,b) = 0\}$ is infinite.  Replacing $P$ with one of its
  irreducible factors that vanishes at $(0,0)$, we may assume that $P$
  is irreducible.  Let $K^* \succ K$ be a highly saturated elementary
  extension.  Because $K$ admits a field topology, it is infinite, and
  thus $\trdeg(K^*/K) \ge 2$.  Then we can find $a_0, b_0 \in K^*$
  algebraically independent over $K$.  By Lemma~\ref{arc}, the
  polynomial $P(z,a_0z^2 + b_0z) \in K^*[z]$ is separable.  Applying
  the Tarski-Vaught criterion to the set $\{(a,b) \in K^* \times K^* :
  P(z,az^2 + bz) \text{ is separable}\}$, we get $(a,b) \in K^2$ such
  that $P(z,az^2 + bz) \in K[z]$ is separable.  Let $Q(x,y) =
  P(x,ax^2+bx+y)$.  Then $Q(x,0) = P(x,ax^2+bx)$, which is separable.
  Moreover, $Q(0,0) = P(0,0) = 0$, so $0$ is a root of $Q(x,0)$.
  Since $P$ and $Q$ are related by a change of variables, $Q$ has
  infinitely many roots if and only if $P$ does.  Therefore, replacing
  $P$ with $Q$, we may assume $P(x,0)$ is separable.

  Because $P(x,0)$ is separable, there is a neighborhood $U \ni 0$
  such that if $b \in U$, then $P(x,b)$ has a root, by
  bt-henselianity.  Thus $P$ has infinitely many roots.
\end{proof}

\section{Bronze and silver t-henselianity via dimension theory} \label{via-dims}
Let $K$ be a field.  As in the previous section, let $K[X]^d_1$ denote
the set of monic polynomials of degree $d$.  Let $K[X]^{<d}$ denote
the set of polynomials of degree $< d$.  Both $K[X]^d_1$ and
$K[X]^{<d}$ can be identified with $K^d$.
\begin{definition}
  Let $\star : K[X]^{<d} \times K[X]^d_1 \to K[X]^d_1$ be the map such
  that if
  \begin{equation*}
    Q(X) = (X-r_1) \cdots (X-r_d) \text{ for some } r_1,\ldots,r_d \in K^\alg,
  \end{equation*}
  then
  \begin{equation*}
    (P \star Q)(X) = (X - P(r_1)) \cdots (X - P(r_d)).
  \end{equation*}
  In other words, $P \star Q$ is the monic polynomial of degree $d$
  obtained by applying $P$ to the roots of $Q$.
\end{definition}
Using elementary symmetric polynomials, one sees that the map $(P,Q)
\mapsto P \star Q$ is given by polynomials, which is incidentally why
$P \star Q$ lands in $K[X]^d_1$ rather than $K^{\alg}[X]^d_1$.  As a
polynomial map, $(P,Q) \to P \star Q$ is continuous with respect to
any field topology on $K$.
\begin{remark} \phantomsection \label{star-facts}
  Fix a \emph{separable} polynomial $Q_0 \in K[X]^d_1$.
  \begin{enumerate}
  \item The map
    \begin{align*}
      K[X]^{<d} &\to K[X]^d_1 \\
      P &\mapsto P \star Q_0
    \end{align*}
    is finite-to-one.  Indeed, suppose we know $P \star Q_0$.  Let
    $r_1,\ldots,r_d$ be the roots of $Q_0$ (in $K^{\alg}$), and
    $s_1,\ldots,s_d$ be the roots of $P \star Q_0$.  Then $P$ is
    determined by the induced map $\{r_1,\ldots,r_d\} \to
    \{s_1,\ldots,s_d\}$, because the $r_i$ are pairwise distinct and
    $\deg P < d$.  There are only finitely many ($\le d^d$)
    possibilities for this map.
%%   \item If $Q_0 \in K[X]^d_1$ is separable, then the set
%%     \begin{equation*}
%%       \{P \in K[X]^{<d} : P \star Q_0 \text{ isn't separable}\}
%%     \end{equation*}
%%     is contained in a $(d-1)$-dimensional algebraic subvariety of
%%     $K[X]^{<d}$ (a hypersurface).  To prove this, we can replace $K$
%%     with $K^\alg$ and assume $K$ is algebraically closed.  Then the
%%     set in question is a union of finitely many sets of the form
%%     \begin{equation*}
%%       \{P \in K[X]^{<d} : P(r_i) = P(r_j)\}
%%     \end{equation*}
%%     where $r_i, r_j$ are distinct roots of $Q_0$.  Every such set is a
%%     hypersurface.
  \item If $Q_0$ and $P \star Q_0$ are separable, then $Q_0$ has a
    $K$-rational root if and only if $P \star Q_0$ has a $K$-rational
    root.  One direction is obvious: if $r$ is a root of $Q_0$, then
    $P(r)$ is a root of $Q_0$.  For the other direction, note that $P$
    induces a surjection from the roots of $Q_0$ to the roots of $P
    \star Q_0$.  Since both polynomials are monic separable of degree
    $d$, these two sets both have size $d$, so the surjection is a
    bijection.  If $r$ is a $K$-rational root of $P \star Q_0$, then
    there is a unique root $s$ of $Q_0$ lifting $r$.  Working in
    $K^{\alg} \models \ACF$, it follows that $s \in \dcl(K)$.  On the
    other hand, $s \in K^{\sep}$ because $Q_0$ is separable.  Thus $s
    \in \dcl(K) \cap K^{\sep} = K$.  \qedhere.
  \end{enumerate}
\end{remark}
Fix a monster model $\Mm$ of some theory, and a definable field $K$.
For the remainder of this section, make the following assumption:
\begin{assumption} \label{A}
  For every definable set $X \subseteq K^n$ there is an associated
  dimension $\dim(X) \in \Nn \cup \{-\infty\}$ satisfying the
  following properties:
  \begin{enumerate}
  \item $\dim(X) = -\infty \iff X = \varnothing$.
  \item $\dim(K) > 0$.
  \item If $\dim(X) > 0$, then $X$ is infinite (but the converse need
    not hold).
  \item $\dim(X \times Y) = \dim(X) + \dim(Y)$.
  \item $\dim(X \cup Y) = \max(\dim(X),\dim(Y))$.
  \item If $f : X \to Y$ is a definable bijection, then $\dim(X) =
    \dim(Y)$.  More generally, if $f : X \to Y$ is a definable
    surjection with finite fibers, then $\dim(X) = \dim(Y)$.
  \end{enumerate}
  Moreover, there is a definable field topology $\tau$ on $K$ with the
  following property:
  \begin{enumerate}[resume]
  \item If $X \subseteq K^n$ then $\dim(X) = \dim(K^n)$ if and only if
    $X$ has non-empty interior.
  \end{enumerate}
\end{assumption}
\begin{remark}
  If $X \subseteq Y$, then $\dim(X) \le \dim(Y)$, because $\dim(Y) =
  \dim(X \cup Y) = \max(\dim(X),\dim(Y))$.  As a corollary, if $f : X
  \to Y$ is a definable injection, or more generally, a definable
  finite-to-one map, then $\dim(X) = \dim(\im(f)) \le \dim(Y)$.
\end{remark}
\begin{remark}
  If $X \subseteq K^n$, then $\dim(\bd(X)) < \dim(K^n)$, because
  $\bd(X)$ can be written as a union of the definable sets $\cl(X)
  \setminus X$ and $X \setminus \ter(X)$, which both have empty
  interior.
\end{remark}
\begin{remark}
  If $P(X_1,\ldots,X_n) \in K[X_1,\ldots,X_d]$ is non-trivial, and
  $V(P)$ is the set of zeros of $P$, then $\dim(V(P)) < \dim(K^n)$.
  This can be seen in a couple different ways.
  \begin{itemize}
  \item Using the proof of the Noether normalization theorem, one can
    change coordinates via a definable bijection, and arrange for the
    projection
    \begin{align*}
      V(P) &\to K^{n-1} \\
      (x_1,\ldots,x_n) &\mapsto (x_1,\ldots,x_{n-1})
    \end{align*}
    to be finite-to-one.  Then $\dim(V(P)) \le \dim(K^{n-1}) <
    \dim(K^n)$.
  \item More directly, if $\dim(V(P)) = \dim(K^n)$, then $V(P)$ has
    non-empty interior.  Thus there is a non-empty open set $U
    \subseteq K^n$ on which $P$ vanishes.  Then we can find infinite
    subsets $S_1,\ldots,S_n \subseteq K$ such that $P$ vanishes on
    $S_1 \times \cdots \times S_n$.  This easily implies that $P$ is
    identically zero, a contradiction.
  \end{itemize}
\end{remark}
\begin{theorem} \label{A-thm}
  Under Assumption~\ref{A}, $(K,\tau)$ is bronze t-henselian.
\end{theorem}
\begin{proof}
  Let $n = \dim(K[X]^d_1) = \dim(K[X]^{<d}) = \dim(K^d) = d \cdot
  \dim(K)$.  Let $S \subseteq K[X]^d_1$ be the set of separable
  polynomials with at least one root in $K$.  We must show that $S$ is
  open.  Fix some $Q_0 \in S$.  We will show that $S$ is a
  neighborhood of $Q_0$.  The set $\bd(S)$ has $\dim(\bd(S)) < n$.
  Let $S_1 \subseteq K[X]^d_1$ be the set of non-separable
  polynomials.  Then $S_1$ is contained in a hypersurface (use
  discriminants), so $\dim(S_1) < n$.  Since the map $P \mapsto P
  \star Q_0$ is finite-to-one, the set
  \begin{equation*}
    S_2 = \{P \in K[X]^{<d} : P \star Q_0 \in S_1 \cup \bd(S)\}
  \end{equation*}
  has $\dim(S_2) < n$.  Take some $P_0 \in K[X]^{<d} \setminus S_2$.
  Then $P_0 \star Q_0 \notin S_1$, so both $Q_0$ and $P_0 \star Q_0$
  are separable.  Since $Q_0$ has a $K$-rational root, so does $P_0
  \star Q_0$.  Thus $P_0 \star Q_0 \in S$.  Moreover, $P_0 \star Q_0
  \notin \bd(S)$, so instead $P_0 \star Q_0 \in \ter(S)$.  Since the
  $\star$ operation is continuous, the set
  \begin{equation*}
    U = \{Q \in K[X]^d_1 : P_0 \star Q \in S\}
  \end{equation*}
  is a neighborhood of $Q_0$.  Since $S_1$ is closed (use
  discriminants) and $Q_0 \notin S_1$, it follows that $U \setminus
  S_1$ is also a neighborhood of $Q_0$.  It remains to show that $U
  \setminus S_1 \subseteq S$.  Suppose $Q \in U \setminus S_1$.  Then
  $Q$ and $P_0 \star Q$ are separable, and $P_0 \star Q \in S$, so
  $P_0 \star Q$ has a $K$-rational root.  Therefore $Q$ has a
  $K$-rational root, and $Q \in S$.
\end{proof}
To get silver t-henselianity, we need a more powerful tool from
algebraic geometry, and an additional assumption on $\tau$.
\begin{fact}[Zariski's Main Theorem] \label{zmt}
  Let $K$ be a field.  Let $X, Y$ be irreducible, normal
  $K$-varieties.  Let $f : X \to Y$ be a morphism of $K$-varieties
  such that $f$ is birational and $X(K^\alg) \to Y(K^\alg)$ is a
  bijection.  Then $f$ is an isomorphism.
\end{fact}
According to Mumford~\cite[Chapter III, \S9]{red-book}, Fact~\ref{zmt} is part of Zariski's original form
of his main theorem.  We will only need the case where the varieties
are smooth, which is supposedly a consequence of \cite[Chapter III, \S9, Proposition 1]{red-book}.
\begin{remark} \label{ahem}
  Let $V, W$ be irreducible $K$-varieties and $f : V \to W$ be a
  morphism which is dominant ($\im(f)$ is dense in $W$).  Embed $K$
  into a monster model $\Mm$ of ACF.  Take $\ba \in V(\Mm)$ generic over $K$ and
  let $\bb = f(\ba)$.  Then $\bb$ is generic in $W(\Mm)$ over $K$, and we can
  identify $K(V) = K(\ba)$ and $K(W) = K(\bb)$.  Under this
  identification, the map $f^* : K(W) \to K(V)$ is the inclusion
  $K(\bb) \subseteq K(\ba)$.
\end{remark}
\begin{lemma} \label{lem-vw}
  Let $K$ be a field.  Let $V_d$ and $W_d$ be the affine varieties
  \begin{gather*}
    V_d =\{(P,Q,c) : P \in K^{\alg}[X]^{<d}, Q \in K^{\alg}[X]^d_1, c \in
    K^\alg, Q(c) = 0\} \\
    W_d = \{(P,Q,c) : P \in K^{\alg}[X]^{<d}, Q \in K^{\alg}[X]^d_1, c \in
    K^\alg, (P \star Q)(c) = 0\}.
  \end{gather*}
  Let $f : V_d \to W_d$ be the morphism $(P,Q,c) \mapsto (P,Q,P(c))$.
  Let $V^0_d$ and $W^0_d$ be the open subvarieties
  \begin{gather*}
    V^0_d = \{(P,Q,c) \in V_d : Q \text{ and } P \star Q \text{ are separable}\} \\
    W^0_d = \{(P,Q,c) \in W_d : Q \text{ and } P \star Q \text{ are separable}\}.
  \end{gather*}
  Then $f : V^0_d \to W^0_d$ is an isomorphism of $K$-varieties.
\end{lemma}
\begin{proof}
  By Zariski's Main Theorem, it suffices to check the following six underlined statements:
  \begin{enumerate}
  \item \underline{$V^0_d$ is smooth}. It suffices to show that any
    $(P,Q,c) \in V^0_d$ is a smooth point of $V_d$.  Note that $V_d$
    is a hypersurface, the vanishing set of the polynomial $h(P,Q,c) =
    Q(c)$.  To show smoothness, we merely need at least one partial
    derivative of $h$ to be non-zero.  Since $Q(c) = 0$ and $Q$ is
    separable,
    \begin{equation*}
      \frac{\partial h}{\partial c} = \frac{\partial}{\partial c}Q(c)
      \ne 0.
    \end{equation*}
  \item \underline{$W^0_d$ is smooth}.  Similar, using the fact that
    $\frac{\partial}{\partial c}(P \star Q)(c) \ne 0$.
  \item \underline{$f$ is a bijection on $K^\alg$-points}.  This holds
    because if $Q$ and $P \star Q$ are both separable, then $P$
    induces a bijection from the ($K^\alg$-rational) roots of $Q$ to
    the ($K^\alg$-rational) roots of $P \star Q$.
  \item \underline{$V^0_d$ is irreducible}.  As $V^0_d$ is an open
    subvariety of $V_d$, it suffices to show that $V_d$ is
    irreducible.  Note that $V_d$ is a direct product of
    $K^\alg[X]^{<d}$ (which is the irreducible $d$-dimensional affine
    space) and the variety $X = \{(Q,c) \in K^{\alg}[X]^d_1 \times
    K^\alg : Q(c) = 0\}$.  The variety $X$ is
    \begin{equation*}
      \{(y_0,y_1,\ldots,y_{d-1},x) \in (K^\alg)^d : x^d +
      y_{d-1}x^{d-1} + \cdots + y_1x + y_0 = 0\}.
    \end{equation*}
    This is merely affine $d$-space in the coordinates
    $y_1,\ldots,y_{d-1},x$, under the change of variables
    \begin{equation*}
      y_0 = -y_1x  -y_2x^2 - \cdots - y_{d-1}x^{d-1} - x^d.
    \end{equation*}
    In particular, $X$ is geometrically irreducible.  A product of two
    geometrically irreducible varieties is irreducible.
  \item \underline{$W^0_d$ is irreducible}.  This holds because it is
    the image of the irreducible variety $V^0_d$ under the morphism
    $f$.
  \item \underline{$f$ is birational}.  We use Remark~\ref{ahem}.
    Embed $K$ into a monster model $\Mm$ of ACF.  Take $(P,Q,c)$
    generic in $V^0_d$.  Then $K(P,Q,c)$ is the function field
    $K(V^0_d)$.  The image $f(P,Q,c) = (P,Q,P(c))$ is generic in
    $W^0_d$, and $K(P,Q,P(c))$ is the function field $K(W^0_d)$.  The
    natural map $K(W^0_d) \to K(V^0_d)$ induced by $f$ is the
    inclusion $K(P,Q,P(c)) \subseteq K(P,Q,c)$.  We want this
    inclusion to be the identity, and so it suffices to show that $c
    \in K(P,Q,P(c))$.  Note that $Q$ is a separable polynomial with
    coefficients in $K(P,Q,P(c))$, and $c$ is a root of $Q$.
    Therefore $c \in K(P,Q,P(c))^{\sep}$.  Additionally, working in
    $\Mm \models \ACF$, the value $c$ is the unique root of $Q$
    mapping to $P(c)$ under the polynomial map $P$.  Therefore, $c \in
    \dcl(P,Q,P(c))$.  So
    \begin{equation*}
      c \in K(P,Q,P(c))^{\sep} \cap \dcl(K(P,Q,P(c))) = K(P,Q,P(c)).  \qedhere
    \end{equation*}
  \end{enumerate}
\end{proof}
Now return to the setting of Assumption~\ref{A}.
\begin{assumption} \label{B}
  If $U$ is a non-empty open subset of $K^n$ and $f : U
  \rightrightarrows K^m$ is a $k$-correspondence, then the set
  \begin{equation*}
    B = \{\ba \in K^n : f \text{ isn't continuous at } \ba\}
  \end{equation*}
  has $\dim(B) < \dim(K^n)$.
\end{assumption}
See Definition~\ref{corr-def} below for the definition of
``$k$-correspondences'' and their continuity.
\begin{theorem} \label{silver-theorem}
  Under Assumptions~\ref{A} and \ref{B}, the topology $\tau$ is
  st-henselian.
\end{theorem}
\begin{proof}
  Let $n = \dim(K[X]^d_1) = \dim(K[X]^{<d}) = \dim(K^d) = d \cdot
  \dim(K)$.  Let $S \subseteq K[X]^d_1$ be the open subset of
  separable polynomials.  For $0 \le k \le d$, let $S_k$ be the subset
  of $P \in S$ such that $P$ has exactly $k$ roots over $K$.  Let $f_k
  : S_k \rightrightarrows K$ be the $k$-correspondence sending a
  polynomial to its set of roots.  To prove st-henselianity, it
  suffices to show that each set $S_k$ is open, and each
  correspondence $f_k$ is continuous.
  
  Let $S^0_k$ be the set of $P \in \ter(S_k)$ such that $f_k$ is
  continuous on a neighborhood of $P$.  Then $S^0_k$ is an open subset
  of $S_k$.  We must show that $S^0_k = S_k$.  First note that
  $\dim(S_k \setminus S^0_k) < n$ by Assumptions~\ref{A} and \ref{B}.
  Fix some $Q_0 \in S_k$.  We will show that $Q_0 \in S^0_k$.  Since
  $Q_0$ is separable, the map
  \begin{align*}
    K[X]^{<d} &\to K[X]^d_1 \\
    P &\mapsto P \star Q_0
  \end{align*}
  is finite-to-one.  As in the proof of bt-henselianity, the set
  $K[X]^d_1 \setminus S$ of non-separable polynomials has dimension $<
  n$.  Since $P \mapsto P \star Q_0$ is finite to one, the set
  \begin{equation*}
    X = \{P \in K[X]_{<d} : P \star Q_0 \notin S \text{ or } P \star
    Q_0 \in S_k \setminus S^0_k\}
  \end{equation*}
  has $\dim(X) < n$.  Fix some $P_0 \in K[X]_{<d} \setminus X$.  Then
  $P_0 \star Q_0 \in S$, so $Q_0$ and $P_0 \star Q_0$ are both
  separable.  Then $P_0$ induces a bijection between the $K$-rational
  roots of $Q_0$ and those of $P_0 \star Q_0$.  Since $Q_0$ has
  exactly $k$ roots over $K$, so does $P_0 \star Q_0$, and so $P_0
  \star Q_0 \in S_k$.  Then $P_0 \star Q_0 \in S^0_k$ (or else $P_0
  \in X$).  In particular, $f_k$ is continuous on a neighborhood of
  $P_0 \star Q_0$.

  Let $r_1, \ldots, r_k$ be the $K$-rational roots of $P_0 \star Q_0$,
  so $\{r_1,\ldots,r_k\} = f_k(P_0 \star Q_0)$.  Because $f_k$ is
  continuous on a neighborhood of $Q_0$, there is a neighborhood $U
  \subseteq S^0_k$ of $P_0 \star Q_0$ and continuous functions
  $g_1,\ldots,g_k : U \to K$ such that
  \begin{equation*}
    f(P) = \{g_1(P),\ldots,g_k(P)\} \text{ for } P \in U.
  \end{equation*}
  As the map $Q \mapsto P_0 \star Q$ is continuous (in fact,
  polynomial), the set \[U' = \{Q \in K[X]^d_1 : Q \in S \text{ and }
  P_0 \star Q \in U\}\] is a neighborhood of $Q_0$, and the
  compositions
  \begin{align*}
    h_i : U' &\to K \\
    h_i(Q) &= g_i(P_0 \star Q)
  \end{align*}
  are continuous for each $i$.  Then $\{h_1(Q),\ldots,h_k(Q)\}$ is the
  set of $K$-rational roots of $P_0 \star Q$ for any $Q \in U'$.
  Moreover, $Q$ and $P_0 \star Q$ are separable for $Q \in U'$.  It
  follows that the $K$-rational roots of $Q$ are
  $\{\eta_1(Q),\ldots,\eta_k(Q)\}$, where $\eta_i(Q)$ is determined by
  the fact that $P_0(\eta_i(Q)) = h_i(Q)$.  In particular, $Q \in S_k$
  for each $Q \in U'$, and so the neighborhood $U' \ni Q_0$ shows $Q_0
  \in \ter(S_k)$.  To complete the proof that $Q_0 \in S^0_k$, it
  remains to show that $f_k$ is continuous on the neighborhood $U' \ni
  Q_0$.  Equivalently, we must show that the functions $\eta_i : U'
  \to K$ are continuous.
  
  Let $V^0_d$ and $W^0_d$ be as in Lemma~\ref{lem-vw}.  Then we have a
  continuous map
  \begin{align*}
    U' &\to W^0_d(K) \\
    Q &\mapsto (P_0,Q,h_i(Q))
  \end{align*}
  for each $i$.  Because the morphism $V^0_d \to W^0_d$ is an
  isomorphism of varieties, the map $V^0_d(K) \to W^0_d(K)$ and its
  inverse are continuous (with respect to $\tau$).  Therefore the
  following composition is continuous:
  \begin{gather*}
    U' \to W^0_d(K) \stackrel{\cong}{\to} V^0_d(K) \\
    Q \mapsto (P_0,Q,h_i(Q)) \mapsto (P_0,Q,\eta_i(Q)).
  \end{gather*}
  Then each $\eta_i$ is continuous on $U'$, as desired.  Stepping back
  through the proof, the correspondence $Q \mapsto
  \{\eta_1(Q),\ldots,\eta_k(Q)\} = f_k(Q)$ is continuous on $U'$,
  ensuring that $U' \subseteq S^0_k$ and $Q_0 \in S^0_k$.  As $Q_0 \in
  S_k$ was arbitrary, we see that $S^0_k = S_k$.  Thus $S_k$ is open,
  and $f_k : S_k \rightrightarrows K$ is continuous, which implies
  that $\tau$ is st-henselian.
\end{proof}

% PART 2: the visceral case

\section{Nice t-minimal theories and the independent neighborhoods property} \label{new-insert}

\begin{definition}[{\cite[Definition~4.21]{wj-visc-1}}]
  A t-minimal theory $T$ is \emph{nice} if correspondences are
  generically continuous on open subsets of $\Mm^n$, in the following
  sense:
  \begin{itemize}
  \item If $U \subseteq \Mm^n$ is a definable open set and $f : U
    \rightrightarrows \Mm^m$ is a definable $k$-correspondence, then
    there is a definable open set $U_0 \subseteq U$ such that $\dim(U
    \setminus U_0) < n$ and $f \restriction U_0$ is continuous.
  \end{itemize}
\end{definition}
Visceral theories are nice \cite[Proposition~3.5]{wj-visc-1}.  It is
well-known that o-minimal and C-minimal theories are nice as well.  An
important generalization of all these cases is t-minimal theories with
the following technical condition:
\begin{definition}[{\cite[Definition~3.1]{castle-hasson}}]
  A t-minimal theory $T$ has the \emph{independent neighborhoods
  property} if for any finite tuple $\ba \in \Mm^n$ and small set $B
  \subseteq \Mm$ and neighborhood $U \ni \ba$, there is $B' \subseteq
  \Mm$ and a $B'$-definable neighborhood $U' \ni \ba$ such that
  \begin{gather*}
    B \subseteq B' \subseteq \Mm \\
    \dim(\ba/B') = \dim(\ba/B) \tag{$\ast$} \\
    \ba \in U \subseteq U'.
  \end{gather*}
\end{definition}
The intuition for ($\ast$) is that $B'$ is ``independent'' from $\ba$
in some sense, and then the $B'$-definable neighborhood $U'$ is also
``independent''.

Castle and Hasson show that o-minimal, weakly o-minimal and C-minimal
theories have the independent neighborhoods property
\cite[Corollary~A.15]{castle-hasson}, and visceral theories do too
\cite[Appendix~B]{castle-hasson}.  More importantly, Castle and Hasson
show that the independent neighborhoods property implies generic
continuity of functions.  The same proof works for correspondences.
For completeness, we include the proof:
\begin{theorem}
  Let $T$ be a t-minimal theory with the independent neighborhoods
  property.  Then $T$ is nice.  In other words, if $U \subseteq \Mm^n$
  is a definable open set and $f : U \rightrightarrows \Mm^m$ is a
  definable $k$-correspondence, then $f$ is continuous on an open set
  $U_0 \subseteq U$ with $\dim(U \setminus U_0) < n$.
\end{theorem}
\begin{proof}
  Let $B$ be the set of points in $U$ where $f$ is discontinuous.  If
  $\dim(B) < n$ then the closure $\overline{B}$ also has dimension
  less than $n$ \cite[Corollary~2.10, Theorem~2.37(4)]{wj-visc-1}, and
  we can take $U_0 = U \setminus B$.  If $\dim(B) = n$ then $B$ has
  non-empty interior \cite[Theorem~2.37(4)]{wj-visc-1}; replacing $U$
  with the interior $\ter(B)$, we can assume $f$ is nowhere continuous
  on $U$.  We will get a contradiction.

  Take a small set $C$ over which everything is defined.  Take $\ba
  \in U$ with $\dim(\ba/C) = n$.  We will prove that $f$ is continuous
  at $\ba$.  Let $f(\ba) = \{\bb_1,\ldots,\bb_k\}$.  For each $i$, let
  $V_i$ be a neighborhood of $\bb_i$, chosen so that $V_1,\ldots,V_k$
  are pairwise disjoint.  Let $S$ be the set of $\bx \in U$ such that
  $f(\bx)$ contains exactly one point from each $V_i$.  Then $\ba \in
  S$, and we must show that $S$ is a neighborhood of $\ba$.

  By the independent
  neighborhoods property (see \cite[Lemma~3.2]{castle-hasson}), we can find $C' \supseteq C$ and
  $C'$-definable neighborhoods $V'_i \ni \bb_i$ with
  \begin{gather*}
    V'_i \subseteq V_i \\
    \dim(\ba,\bb_1,\ldots,\bb_k/C') = \dim(\ba,\bb_1,\ldots,\bb_k/C).
  \end{gather*}
  Since $\bb_1,\ldots,\bb_k$ are algebraic over $C'\ba$, we can
  rewrite this second line as
  \begin{equation*}
    \dim(\ba/C') = \dim(\ba/C).
  \end{equation*}
  Let $S'$ be the set of points $\bx \in U$ such that $f(\bx)$
  contains exactly one point in each $V'_i$.  Since $f$ is a
  $k$-correspondence, $S' \subseteq S$.  Note also that $\ba \in S'$,
  and $S'$ is $C'$-definable.  Then $\bd(S')$ is $C'$-definable and
  has dimension less than $n = \dim(\ba/C')$, so $\ba \notin \bd(S')$.
  It follows that $\ba \in \ter(S')$.  That is, $S'$ is a neighborhood
  of $\ba$.  The larger set $S \supseteq S'$ is also a neighborhood of
  $\ba$, completing the proof.
\end{proof}
Essentially all the theorems of \cite{wj-visc-1} generalize from
visceral theories to nice t-minimal theories
\cite[Remark~4.22]{wj-visc-1}.  Similarly, the theorems about visceral
theories in \S\ref{defman}--\ref{atdf} below all generalize to nice
t-minimal theories, so we will work in this greater generality.

\section{Definable manifolds} \label{defman}
Work in a monster model $\Mm$ of a visceral theory, or more generally,
a nice t-minimal theory in the sense of the previous section.  We use
the following definitions and conventions from
\cite{simonWalsberg,wj-visc-1}:
\begin{definition} \label{corr-def}
  If $X$ and $Y$ are sets, a \emph{$k$-correspondence} $f : X \rra Y$
  is a function assigning to each $x \in X$ a $k$-element subset $f(x)
  \subseteq Y$.  A \emph{correspondence} is a $k$-correspondence for
  some $k \ge 1$.  The \emph{graph} of a correspondence $f : X \rra Y$
  is the set
  \begin{equation*}
    \Gamma(f) = \{(x,y) \in X \times Y : y \in f(x)\}.
  \end{equation*}
  If $X$ and $Y$ are topological spaces, then $f$ is \emph{continuous}
  if for any $a \in X$, there is a neighborhood $a \in U \subseteq X$
  on which $f$ has the form $f(x) = \{g_1(x),\ldots,g_k(x)\}$ for some
  continuous functions $g_1,\ldots,g_k : U \to Y$.  If $X$ and $Y$ are
  definable sets, then $f$ is \emph{definable} if $\Gamma(f)$ is
  definable as a subset of $X \times Y$.
\end{definition}
\begin{definition}
  A \emph{cell} is a set of the form $\sigma(\Gamma(f)) \subseteq
  \Mm^n$ where $U \subseteq \Mm^i$ is a definable open set, $f : U
  \rra \Mm^{n-i}$ is a continuous definable correspondence, and
  $\sigma$ is a coordinate permutation.
\end{definition}
\begin{definition}
  A \emph{definable topological space} is a topological space
  $(X,\tau)$ where $X$ is a definable set, $\tau$ is a topology on
  $X$, and there is a definable family $\{U_y\}_{y \in Y}$ such that
  $\{U_y : y \in Y\}$ is a basis for the topology $\tau$.
\end{definition}
The following definitions are essentially from
\cite[Definition~4.1]{admissible}.
\begin{definition}
  Let $X$ be a definable topological space.
  \begin{enumerate}
  \item $X$ is \emph{locally Euclidean} if for every point $p \in X$,
    there is a definable open neighborhood $p \in U \subseteq X$
    which is definably homeomorphic to an open definable subspace
    of $\Mm^n$ for some $n$ depending on $p$.
%%   \item $X$ is \emph{locally definable} if for every point $p \in X$,
%%     there is an interpretable open neighborhood $p \in U \subseteq X$
%%     which is interpretably homeomorphic to a definable subspace of
%%     $\Mm^n$ for some $n$ depending on $p$.
  \item $X$ is a \emph{definable manifold} if $X = \bigcup_{i=1}^n
    U_i$ for some definable open subsets $U_i \subseteq X$, each
    of which is definably homeomorphic to an open subset of a
    cell.
  \item $X$ is a \emph{definable semimanifold} if $X = \bigcup_{i=1}^n
    U_i$ for some definable open subsets $U_i \subseteq X$, each
    of which is definably homeomorphic to a definable subspace of
    $\Mm^k$ for some $k$ depending on $i$.
  \end{enumerate}
\end{definition}
%% \textsc{Actually}, I should assume that (pre-)manifolds are definable,
%% rather than interpretable, for now.
\begin{remark}
  \begin{enumerate}
  \item Cells and open subsets of cells are locally Euclidean.
    Consequently, definable manifolds are locally Euclidean.    
  \item The definitions of ``definable manifold'' and ``locally
    Euclidean'' are similar, except that for ``definable
    manifold'' we require a finite atlas, while the atlas in
    ``locally Euclidean'' can be infinite.  In particular, definable manifolds are
    locally Euclidean, but locally Euclidean spaces need not be
    definable manifolds.
  \item Any definable open set $U \subseteq \Mm^n$ is a cell, and
    therefore a definable manifold.
%%   \item The four concepts are related as follows:
%%     \begin{equation*}
%%       \xymatrix{
%%           \text{Definable manifolds} \ar@{}[r]|-*[@]{\subseteq} \ar@{}[d]|-*[@]{\subseteq} & \text{Definable pre-manifolds} \ar@{}[d]|-*[@]{\subseteq} \\ \text{Locally euclidean spaces} \ar@{}[r]|-*[@]{\subseteq} & \text{Locally definable spaces.}
%%       }
%%     \end{equation*}
%%   \item If $X$ is locally definable (bzw.\@ a definable pre-manifold)
%%     and $Y$ is an interpretable subspace, then $Y$ is locally
%%     definable (bzw.\@ a definable pre-manifold).
  \item If $X$ is a definable semimanifold, and $Y \subseteq X$ is a
    definable subspace, then $Y$ is a definable semimanifold.
  \item If $X$ is a definable manifold (bzw.\@ locally Euclidean), and
    $Y$ is an open definable subspace, then $Y$ is a definable
    manifold (bzw.\@ locally Euclidean).\footnote{This is why we use ``open subset of a cell'' rather than ``cell'' in the definition of ``definable manifold''.  It's not clear whether any open subset of a cell can be covered by finitely many open sets homeomorphic to cells.}
%%   \item It is a general fact that if $X$ is an interpretable set and
%%     $X$ is covered by the image of finitely many interpretable
%%     injections $f_i : D_i \to X$, with $D_i$ definable, then $X$ is in
%%     interpretable bijection with a definable set, and so $X$ is
%%     ``definable'' in some sense.  Consequently, any definable
%%     (pre-)manifold is ``definable''.  More precisely, any definable
%%     pre-manifold $X$ has an interpretable homeomorphism to a definable
%%     pre-manifold $X'$ whose underlying set is definable, rather than
%%     interpretable.
  \item Definable semimanifolds are essentially the same thing as the
    ``definable spaces'' in \cite{Peterzil-Steinhorn}.  In our
    context, ``definable space'' sounds too close to ``definable topological space'', so we
    use the term ``definable semimanifold'' instead.  The name is
    meant to remind one of semialgebraic sets, which can have edges
    and corners, just like a semimanifold.
  \item In \cite{Peterzil-Steinhorn} and \cite{admissible}, the
    definition of ``definable manifold'' required each of the open
    sets $U_i$ to be definably homeomorphic to an open subspace of
    $\Mm^n$.  Here, we are requiring something weaker: an open subset
    of a cell.  This weakening is essential if we want definable
    groups to be definable manifolds, even in the well-behaved case of
    ACVF$_{0,0}$.  See Appendix~\ref{not-nice}.
  \end{enumerate}
\end{remark}

\section{Some tame topology on definable (semi)manifolds} \label{tame-man}
If $X$ is a definable topological space and $D \subseteq X$, the
\emph{local dimension} of $D$ at $p \in X$, written $\dim_p(D)$, is
the infimum of $\dim(U \cap D)$ as $U$ ranges over neigbhorhoods of
$p$.
\begin{lemma} \label{local-dim}
  Let $X$ be a definable semimanifold.  If $D \subseteq X$ is a
  definable subset, then $\dim(D) = \max_{p \in D} \dim_p(D)$, where
  $\dim_p(D)$ is the local dimension of $D$ at $p$.
\end{lemma}
\begin{proof}
  Since $X$ is covered by \emph{finitely many} open sets which are
  definably homeomorphic to definable subspaces of $\Mm^n$, we reduce
  to the case where $X$ itself is a definable subspace $X \subseteq
  \Mm^n$.  Then this is \cite[Proposition~3.10]{wj-visc-1}.
\end{proof}
\begin{lemma} \label{loc-eu-dim}
  Let $X$ be a locally Euclidean definable semimanifold and $Y$ be a
  definable subset.  Then $Y$ is dense if and only if $\dim(X
  \setminus Y) < \dim(X)$.
\end{lemma}
\begin{proof}
  Again, we can work locally, and reduce to the case where $X$ is a
  subset of $\Mm^n$, in which case the result is
  \cite[Proposition~4.18]{wj-visc-1}.
\end{proof}
\begin{lemma} \label{bookkeeping-rescue}
  Let $X$ be a definable semimanifold.  There is a definable set $D
  \subseteq \Mm^n$ and a definable continuous bijection $f : D \to X$.
\end{lemma}
% Wow, this saves a LOT of trouble.
\begin{proof}
  We can cover $X$ with finitely many open definable subsets
  $U_1,\ldots,U_\ell$, each of which has a definable homeomorphism to
  a definable subset of a power of $\Mm$.  Let $Y_1,\ldots,Y_m$ be the
  atoms in the boolean algebra generated by $U_1,\ldots,U_\ell$.  Then
  $\{Y_1,\ldots,Y_m\}$ is a definable finite partition of $X$.  Each
  $Y_i$ is a subset of some $U_j$, and therefore has a definable
  homeomorphism to some definable subset $Y'_i$ of a power of $\Mm$.
  We can realize the topological disjoint union $\coprod_{i=1}^m Y'_i$
  as a definable subset $D \subseteq \Mm^n$ for some $n$.  To conclude, take $f$ to be the natural map
  \begin{equation*}
    D \stackrel{\cong}{\to} \coprod_{i=1}^m Y'_i \to \bigcup_{i=1}^m Y_i = X.  \qedhere
  \end{equation*}
\end{proof}
\begin{lemma} \label{new-tame-sane}
  Let $X$ be a definable semimanifold.
  \begin{enumerate}
  \item \label{nts1} If $Y$ is a definable semimanifold and $f : X
    \rightrightarrows Y$ is a definable $k$-correspondence, then the
    set $X_{bad}$ of points where $f$ is not continuous is the
    complement of a dense set.
  \item \label{nts2} If $Y$ is a definable semimanifold and $f : X \to Y$ is a
    definable function, then the set $X_{bad}$ of points where $f$ is
    not continuous is the complement of a dense set.
  \item \label{nts3} If $D \subseteq X$ is a definable subset and $X_{bad}$ is the
    boundary $\bd(D)$, then $X_{bad}$ is the complement of a dense
    set.
  \item \label{nts4} If $X$ is locally Euclidean, then in each of the previous
    points, the set $X_{bad}$ has $\dim(X_{bad}) < \dim(X)$.
  \end{enumerate}
\end{lemma}
\begin{proof}
  \begin{enumerate}
  \item Easy case: $Y$ is a definable subset of $\Mm^n$.  Working
    locally, we reduce to the case where $X$ is a definable subset of
    $\Mm^n$.  Then apply generic continuity of $k$-correspondences in
    $\Mm$ \cite[Theorem~4.13]{wj-visc-1}.

    Hard case: $Y$ is a general definable semimanifold.  In this case,
    use Lemma~\ref{bookkeeping-rescue} to find a definable set
    $\tilde{Y} \subseteq \Mm^n$ and a definable continuous bijection
    $g : \tilde{Y} \to Y$.  Apply the easy case to the composition
    \begin{equation*}
      g^{-1} \circ f : X \rightrightarrows \tilde{Y}
    \end{equation*}
    to see that $g^{-1} \circ f$ is continuous on a dense set.
    Composing with the continuous function $g$, we see that $f$ is
    continuous on a dense set.
  \item Take $k=1$ in the previous point.
  \item Apply the previous point to the characteristic function
    $\chi_D : X \to \{0,1\}$.
  \item Lemma~\ref{loc-eu-dim} translates density into the condition
    on dimensions.  \qedhere
  \end{enumerate}
\end{proof}
\begin{remark}
  It is essential to assume local Euclideanity to get the
  dimension-theoretic statements, because of the ill-behaved examples
  of \cite[Section~5]{wj-visc-1}.  Specifically, there can be
  definable sets $D \subseteq \Mm^n$ with $\dim(\partial D) > \dim(D)$
  (see \cite[Proposition~5.30]{wj-visc-1}).  If $X$ is the definable
  set $\overline{D} = D \cup \partial D$, regarded as a subspace of
  $\Mm^n$, then the pair $D \subseteq X$ contradicts part (\ref{nts4})
  in Lemma~\ref{new-tame-sane}.
\end{remark}
\begin{remark} \label{some-local-eu}
  If $X$ is a non-empty definable semimanifold, there is a point $p
  \in X$ such that $X$ is locally Euclidean at $p$.
\end{remark}
\begin{proof}
  We immediately reduce to the case where $X$ is a definable subspace
  of $\Mm^n$.  Then this is \cite[Theorem~4.17]{wj-visc-1}.
\end{proof}
\begin{lemma} \label{interior-duh}
  Let $X$ be a definable manifold whose local dimension is everywhere
  $k$, so in particular $\dim(X) = k$.  If $Y$ is a definable subset,
  then $\ter(Y) \ne \varnothing \iff \dim(Y) = k$.
\end{lemma}
\begin{proof}
  If $p \in \ter(Y)$, then $\dim_p(Y) = \dim_p(X) = k$ so $\dim(Y) =
  k$.  Conversely, if $\dim(Y) = k$, then $\dim(\bd(Y)) < \dim(X) =
  k$, so $\bd(Y) \not\supseteq Y$ and $\ter(Y) = Y \setminus \bd(Y)$
  is non-empty.
\end{proof}

\section{Topologizing definable groups}

\subsection{Uniqueness}

\begin{proposition} \label{prop-unique}
  Let $(G,\cdot)$ be a definable group.  Then there is at most one
  definable topology $\tau$ on $G$ with the following properties:
  \begin{enumerate}
  \item $(G,\tau)$ is a definable semimanifold.
  \item For every $a \in G$, the left translation
    \begin{align*}
      \lambda_a : G &\to G \\
      \lambda_a(x) &= a \cdot x
    \end{align*}
    is continuous with respect to $\tau$.
  \end{enumerate}
\end{proposition}
\begin{proof}
  Let $k = \dim(G)$.
  Note that the conditions (1) and (2) imply
  \begin{enumerate}
    \setcounter{enumi}{2}
  \item $(G,\tau)$ is locally Euclidean.
  \end{enumerate}
  Indeed, Remark~\ref{some-local-eu} gives \emph{some} point $p \in G$
  such that $(G,\tau)$ is locally Euclidean at $p$.  Since $\tau$ is
  invariant under left translations, it follows that $(G,\tau)$ is
  locally Euclidean everywhere.

  Let $\tau, \tau'$ be two topologies satisfying (1)--(3).  We claim
  that $\tau = \tau'$.  Applying Lemma~\ref{new-tame-sane}(\ref{nts2}) to the
  identity map
  \begin{equation*}
    \id_G : (G,\tau) \to (G,\tau'),
  \end{equation*}
  we get a point $a \in G$ such that $\id_G : (G,\tau) \to (G,\tau')$
  is continuous at $a$.  Applying left translations, it follows that
  $\id_G$ is continuous everywhere.  So every $\tau'$-open set is
  $\tau$-open.  By symmetry, every $\tau$-open set is $\tau'$-open,
  and $\tau = \tau'$.
\end{proof}

\subsection{Existence}
\begin{lemma}
  If $D \subseteq \Mm^n$ is a definable set, there is a countable
  subset $S \subseteq D$ such that if $D' \subseteq D$ is a definable
  subset and $S \subseteq D'$, then $\dim(D') = \dim(D)$.
\end{lemma}
This is implicit in the proof of \cite[Corollary~2.43]{wj-visc-1}, but
we spell out the details:
\begin{proof}
  Let $k = \dim(D)$.  By \cite[Proposition~2.41]{wj-visc-1}, we can
  write $D$ as a union of finitely many sets $\bigcup_{i=1}^m X_i$
  where each $X_i$ has a ``near-injective $k$-projection''
  \cite[Definition~2.40]{wj-visc-1}, meaning that there is a
  coordinate projection $\pi : \Mm^n \to \Mm^k$ such that $X_i \to
  \pi(X_i)$ has finite fibers.  By dimension theory
  \cite[Theorem~2.37(2,5,7)]{wj-visc-1}, one of the $X_i$, say, $X_1$
  has dimension $k$, and then the image $\pi(X_1) \subseteq \Mm^k$ has
  dimension $k$ and has non-empty interior.  Then $\pi(X_1)$ contains
  a box $\prod_{i=1}^k B_i$ where each $B_i$ is a basic open set.
  Take a countable infinite subset $S_i \subset B_i$ for each $i$.
  Let $S$ be the preimage of $\prod_{i=1}^k S_i \subseteq \pi(X_1)$
  under the map $X_1 \to \pi(X_1)$:
  \begin{equation*}
    S = \{a \in X_1 : \pi(a) \in \prod_{i=1}^k S_i\}.
  \end{equation*}
  Then $S$ is countable, since $X_1 \to \pi(X_1)$ has finite fibers
  and $\prod_{i=1}^k S_i$ is countable.  Note that $\pi : S \to
  \prod_{i=1}^k S_i$ is surjective, since $\prod_{i=1}^k S_i \subseteq
  \pi(X_1)$.

  Suppose $D' \subseteq D$ contains $S$.  Then $D' \cap X_1$ contains
  $S$, so $\pi(D' \cap X_1)$ contains $\prod_{i=1}^k S_i$.  Then
  $\pi(D' \cap X_1)$ is ``broad'' in the sense of
  \cite[Definition~2.1]{wj-visc-1}, so $\dim(\pi(D' \cap X_1)) = k$ by
  \cite[Theorem~2.37(4)]{wj-visc-1}.  The map $D' \cap X_1 \to \pi(D'
  \cap X_1)$ is surjective with finite fibers, so $\dim(D' \cap X_1) =
  \dim(\pi(D' \cap X_1)) = k$ by \cite[Theorem~2.37(7)]{wj-visc-1},
  and therefore $\dim(D') = k$.
\end{proof}
\begin{proposition} \label{covers}
  Let $(G,\cdot)$ be a definable group.  Let $U \subseteq G$ be a
  definable subset such that $\dim(U \setminus G) < \dim(G)$.  Then
  finitely many left translates of $U$ cover $G$.
\end{proposition}
\begin{proof}
  Let $S$ be a countable subset of $G$ such that if $D \subseteq G$ is
  definable and $D \supseteq S$, then $\dim(D) = \dim(G)$.
  \begin{claim}
    If $x \in G$, there is $a \in S$ such that $a \cdot x \in U$.
  \end{claim}
  \begin{claimproof}
    Otherwise, for every $a \in S$, we have $a \cdot x \in G \setminus
    U$, or equivalently, $a \in (G \setminus U) \cdot x^{-1}$.  Then
    $S \subseteq (G \setminus U) \cdot x^{-1}$.  By choice of $S$, $(G
    \setminus U) \cdot x^{-1}$ has the same dimension as $G$.  But
    \begin{equation*}
      \dim((G \setminus U) \cdot x^{-1}) = \dim(G \setminus U) <
      \dim(G). \qedhere
    \end{equation*}
  \end{claimproof}
  By the claim, for every $x \in G$, there is $a \in S$ such that $a
  \cdot x \in U$, or equivalently, $x \in a^{-1} \cdot U$.  Therefore
  \begin{equation*}
    G \subseteq \bigcup_{a \in S} a^{-1} \cdot U.
  \end{equation*}
  By saturation, finitely many of the $a^{-1} \cdot U$ cover $G$.
\end{proof}

\begin{lemma} \label{cover-confusion}
  Let $G$ be a definable group.  Let $U \subseteq G$ be a definable
  subset such that finitely many translates of $U$ cover $G$.  Let
  $\tau_0$ be a definable topology on $U$.  Suppose that for any $a
  \in G$, the set $(a \cdot U) \cap U$ is $\tau_0$-open, and the map
  \begin{align*}
    (a^{-1} \cdot U)\cap U &\to (a \cdot U) \cap U \\
    x &\mapsto a \cdot x
  \end{align*}
  is $\tau_0$-continuous.  \textsc{Then}, there is a definable
  topology $\tau$ on $G$ such that
  \begin{enumerate}
  \item $(U,\tau_0) \to (G,\tau)$ is an open embedding, meaning that
    $U$ is $\tau$-open and $\tau_0 = \tau \restriction U$.
  \item Every left translation    \begin{align*}
      \lambda_a : G &\to G \\
      \lambda_a(x) &= a \cdot x
    \end{align*} is
    $\tau$-continous.
  \end{enumerate}
\end{lemma}
\begin{proof}
  Let $\mathcal{B}$ be the family of sets of the form $a \cdot V$
  where $a \in G$ and $V \subseteq U$ is $\tau_0$-open.
  \begin{claim}
    $\mathcal{B}$ is closed under intersection.
  \end{claim}
  \begin{claimproof}
    Suppose $a,b \in G$, and $V, W \subseteq U$ are $\tau_0$-open.
    Then
    \begin{equation*}
      a \cdot V \cap b \cdot W = a \cdot (V \cap a^{-1} \cdot b \cdot
      W).
    \end{equation*}
    By the assumption on $\tau_0$, the set $U \cap a^{-1} \cdot b
    \cdot W$ is $\tau_0$-open.  Then so is
    \begin{equation*}
      V \cap a^{-1} \cdot b \cdot W = V \cap (U \cap a^{-1} \cdot b
      \cdot W).
    \end{equation*}
    So $a \cdot V \cap b \cdot W$ is a left translate of a
    $\tau_0$-open set.
  \end{claimproof}
  Then $\mathcal{B}$ is a basis for some topology $\tau$ on $G$.  It
  is clearly invariant under left translations.  The set $U = 1 \cdot
  U$ is clearly $\tau$-open.  If $V \subseteq U$ is $\tau_0$-open,
  then $V = 1 \cdot V$ is clearly $\tau$-open.  Conversely, if $V
  \subseteq U$ is $\tau$-open, then $V$ is $\tau_0$-open: $V$ is a
  union of sets $a \cdot W \in \mathcal{B}$, and each set $a \cdot W$
  is $\tau_0$-open by the assumption on $\tau_0$.  Thus $\tau_0 = \tau
  \restriction U$.

  Finally, we must show that $\tau$ is a definable topology.  Let
  $\mathcal{B}_0$ be a definable basis for $\tau_0$.  Then $\{a \cdot
  V : a \in G, ~ V \in \mathcal{B}_0\}$ is a definable basis for
  $\tau$.
\end{proof}
\begin{proposition} \label{almost}
  Let $(G,\cdot)$ be a definable group.  Then there is a topology
  $\tau$ such that
  \begin{enumerate}
  \item $(G,\tau)$ is a definable manifold.
  \item $\tau$ is invariant under left translations.
  \end{enumerate}
\end{proposition}
\begin{proof}
  Suppose $G \subseteq \Mm^n$.  By cell decomposition
  \cite[Theorem~3.8]{wj-visc-1}, we can write $G$ as a disjoint union
  $\coprod_{i=1}^m C_i$, where each $C_i \subseteq \Mm^n$ is a cell.
%%   Order the cells so that $C_1,\ldots,C_\ell$ have $\dim(C_i) =
%%   \dim(G)$, and $C_{\ell+1},\ldots,C_m$ have $\dim(C_i) < \dim(G)$.
  Give each cell $C_i$ the topology as a subspace of $\Mm^n$, and let
  $\tau_1$ be the \emph{disjoint union topology} on $G =
  \coprod_{i=1}^m C_i$.  Then $\tau_1$ is a definable topology on $G$,
  making $G$ into a definable manifold.  But the group operations are
  probably not continuous.

  If $a, b \in G$, let $\lambda_{a,b}$ be the unique left translation
  sending $a$ to $b$, i.e., $\lambda_{a,b}(x) = b \cdot a^{-1} \cdot
  x$.  Note that $\lambda_{a,b}$ is the inverse of $\lambda_{b,a}$.
  Say that $a$ and $b$ are \emph{equivalent}, written $a \sim b$, if
  $\lambda_{a,b}$ is $\tau_1$-continuous at $a$ and $\lambda_{b,a}$ is
  $\tau_1$-continuous at $b$.
  \begin{claim}
    The relation $\sim$ is a definable equivalence relation on $G$.
  \end{claim}
  \begin{claimproof}
    Definability and symmetry are clear.  Reflexivity is true because
    $\lambda_{a,a}$ is the identity map, which is continuous
    everywhere.  For transitivity, suppose $a \sim b$ and $b \sim c$.
    Then $\lambda_{a,c} = \lambda_{b,c} \circ \lambda_{a,b}$.  Since
    $\lambda_{a,b}$ is continuous at $a$ and $\lambda_{b,c}$ is
    continuous at $b = \lambda_{a,b}(a)$, it follows that
    $\lambda_{a,c}$ is continuous at $a$.  Similarly, $\lambda_{c,a}$
    is continuous at $c$.
  \end{claimproof}
  Take a small model $M$ defining $(G,\cdot)$ and $\tau_1$.  Let $k =
  \dim(G)$.  Recall that $\dim(G \times G) = 2k$.
  \begin{claim} \label{ind-claim}
    If $(a,b) \in G \times G$ and $\dim(a,b/M) = 2k$, then $a \sim b$.
  \end{claim}
  \begin{claimproof}
    Let $c = b \cdot a^{-1}$, so that $\lambda_{a,b}(x) = c \cdot x$.
    Since $\dcl(M,a,c) = \dcl(M,a,b)$, we have $\dim(a,c/M) =
    \dim(a,b/M) = 2k$.  By subadditivity of dimension
    \cite[Proposition~2.31]{wj-visc-1},
    \begin{equation*}
      2k = \dim(a,c/M) \le \dim(a/Mc) + \dim(c/M) \le \dim(G) +
      \dim(G) = 2k,
    \end{equation*}
    and so $\dim(a/Mc) = k$.  The set $S$ of points where
    $\lambda_{a,b} : G \to G$ is $\tau_1$-discontinuous is
    $Mc$-definable, and it has dimension $< k$ by
    Lemma~\ref{new-tame-sane}(\ref{nts2}).  Then $a \notin S$ because $\dim(a/Mc) =
    k > \dim(S)$, so $\lambda_{a,b}$ is $\tau_1$-continuous at $a$.  A
    similar argument shows that $\lambda_{b,a}$ is $\tau_1$-continuous
    at $b$.
  \end{claimproof}
  Take some $(a_0,b_0) \in G \times G$ with $\dim(a_0,b_0/M) = \dim(G
  \times G) = 2k$.  Let $C$ be the $\sim$-equivalence class of $a_0$.
  By the subadditivity of dimension, $\dim(b_0/Ma_0) = k$.  Since
  $b_0$ belongs to the $Ma_0$-definable set $C$, we must have $\dim(C)
  = k$.  Then $\dim(G \setminus C) < k$, or else we can find $(a,b)$
  generic in $C \times (G \setminus C)$, and Claim~\ref{ind-claim}
  gives $a \sim b$, a contradiction.  So
  \begin{gather*}
    \dim(C) = k \\
    \dim(G \setminus C) < k.
  \end{gather*}
  Let $U$ be the $\tau_1$-interior of $C$.  Then $\dim(C \setminus U)
  < k$ by Lemma~\ref{new-tame-sane}(\ref{nts3}), and so
  \begin{gather*}
    \dim(U) = k \\
    \dim(G \setminus U) < k.
  \end{gather*}
  Moreover, $U$ is $\tau_1$-open (so $U$ is a definable manifold), and
  $a,b \in U \implies a,b \in C \implies a \sim b$.
  \begin{claim}
    For any $c \in G$, the set $c \cdot U \cap U$ is $\tau_1$-open,
    and the map
    \begin{align*}
      (c^{-1} \cdot U)\cap U &\to (c \cdot U) \cap U \\
      x &\mapsto c \cdot x
    \end{align*}
    is $\tau_1$-continuous.
  \end{claim}
  \begin{claimproof}
    Suppose $a \in c^{-1} \cdot U \cap U$.  Let $b = c \cdot a \in c
    \cdot U \cap U$.  Then $a,b \in U$, so $a \sim b$, and the map
    $\lambda_{a,b}(x) = c \cdot x$ is $\tau_1$-continuous at $a$.  In
    particular, if $x$ is sufficiently $\tau_1$-close to $a$, then $c
    \cdot x \in U$, i.e., $x \in c^{-1} U$.  We have shown that
    $c^{-1} \cdot U \cap U$ is $\tau_1$-open and $x \mapsto c \cdot x$
    is $\tau_1$-continuous on it.
  \end{claimproof}
  By Proposition~\ref{covers}, finitely many left translates of $U$
  cover $G$.  Then Lemma~\ref{cover-confusion} gives a definable
  topology $\tau$ on $G$, invariant under left translation, with $U$
  as a $\tau$-open set, extending $\tau_1 \restriction U$.  If $G =
  \bigcup_{i=1}^m a_i \cdot U$, then each subspace $a_i \cdot U$ is
  definably homeomorphic to $(U,\tau_1 \restriction U)$ via
  translation.  Since $(U,\tau_1 \restriction U)$ is a definable
  manifold, so is each subspace $a_i \cdot U$.  Then since $G$ is
  covered by finitely many open subsets that are definable manifolds,
  $G$ is a definable manifold.
\end{proof}
\begin{proposition} \label{prop-exists}
  Let $(G,\cdot)$ be a definable group.  Then there is a topology
  $\tau$ such that
  \begin{enumerate}
  \item $(G,\tau)$ is a definable manifold.
  \item $\tau$ is a group topology.
  \end{enumerate}
\end{proposition}
\begin{proof}
  Let $\tau$ be the left-invariant topology from
  Proposition~\ref{almost}.  If $\lambda$ is a left translation of $G$
  and $\rho$ is a right translation of $G$, then
  \begin{equation*}
    \lambda(\rho(\tau)) = \rho(\lambda(\tau)) = \rho(\tau)
  \end{equation*}
  because left and right translations commute.  Therefore $\rho(\tau)$
  is invariant under left translations.  By the uniqueness in
  Proposition~\ref{prop-unique}, $\rho(\tau) = \tau$.  Thus $\tau$ is
  right-invariant, in addition to being left-invariant.  Equivalently,
  right translations are continuous.

  Let $i : G \to G$ be the inverse map.  For any left translation
  $\lambda$, there is a right translation $\rho$ such that $\lambda
  \circ i = i \circ \rho$.  Then
  \begin{equation*}
    \lambda(i(\tau)) = i(\rho(\tau)) = i(\tau),
  \end{equation*}
  so $i(\tau)$ is left-invariant.  By another application of
  uniqueness (Proposition~\ref{prop-unique}), $i(\tau) = \tau$,
  meaning that the inverse map is continuous.

  Let $m : G^2 \to G$ be the multiplication map $m(x,y) = x \cdot y$.
  By generic continuity for definable manifolds
  (Lemma~\ref{new-tame-sane}(\ref{nts2})), there is at least one point $(a,b) \in
  G^2$ such that $m$ is continuous at $(a,b)$.  For any $a_0, b_0$,
  the map
  \begin{equation*}
    m(x,y) = x \cdot y = a_0 \cdot m(a_0^{-1} \cdot x, x \cdot
    b_0^{-1}) \cdot b_0
  \end{equation*}
  is continuous at $(a_0 \cdot a, b \cdot b_0)$, by continuity of $m$
  at $(a,b)$ and continuity of left and right translations everywhere.
  Letting $a_0,b_0$ range over $G$, we see that $m$ is everywhere
  continuous.  Thus $\tau$ is a group topology.
\end{proof}
We now have a uniqueness result for a big class of topologies
(Proposition~\ref{prop-unique}) and an existence result for a small
class of topologies (Proposition~\ref{prop-exists}).  Combining these,
we get the following two theorems:
\begin{theorem} \label{thm-unique}
  If $(G,\cdot)$ is a definable group, there is a unique definable
  topology $\tau$ on $G$ such that
  \begin{enumerate}
  \item $(G,\tau)$ is a definable manifold.
  \item $(G,\cdot,\tau)$ is a topological group.
  \end{enumerate}
\end{theorem}
\begin{theorem} \label{thm-improve}
  Let $\tau$ be a definable topology on $G$ such that
  \begin{enumerate}
  \item $(G,\tau)$ is a definable semimanifold.
  \item Every left translation $\lambda(x) = a \cdot x$ is
    $\tau$-continuous.
  \end{enumerate}
  Then $(G,\tau)$ is a definable manifold, and $(G,\cdot,\tau)$ is a
  topological group.
\end{theorem}
\begin{definition} \label{cantop0}
  The \emph{canonical topology} on a definable group $(G,\cdot)$ is
  the topology $\tau$ from Theorem~\ref{thm-unique}.
\end{definition}
\begin{remark} \label{oh-its}
  Definable semimanifolds are always $T_1$, because the $T_1$ property
  can be checked locally, and the topology on $\Mm^n$ is $T_1$.  For
  group topologies, Hausdorff is equivalent to $T_1$.  Therefore, the
  canonical topology is Hausdorff.
\end{remark}
\begin{remark} \label{product-obvious}
  Let $\tau_G$ denote the canonical topology on $G$.  A product of two
  definable manifolds is a definable manifold, so the product topology
  $\tau_G \times \tau_H$ is a manifold topology on the product group
  $G \times H$, and also a group topology of course.  By the
  uniqueness of the canonical topology, $\tau_G \times \tau_H$ must be
  $\tau_{G \times H}$.
\end{remark}
\begin{remark} \label{interior-duh-2}
  Let $G$ be a $d$-dimensional definable group and $X \subseteq G$ be
  a definable subgroup.  Then $X$ has non-empty interior in $\tau_G$
  if and only if $\dim(X) = d$.  Indeed, $G$ must have constant local
  dimension (because translations are homeomorphisms), so
  Lemma~\ref{interior-duh} applies.
\end{remark}
\subsection{Homomorphisms}
\begin{theorem} \label{continuity}
  Let $G, H$ be definable groups, and let $f : G \to H$ be a definable
  homomorphism.  Then $f$ is continuous, with respect to the canonical
  topologies on $G$ and $H$.
\end{theorem}
\begin{proof}
  By Lemma~\ref{new-tame-sane}, there is \emph{some} $a \in G$ such that
  $f$ is continuous at $a$.  Then for any $b \in G$, the map
  \begin{equation*}
    f(x) = f(b) \cdot f(b^{-1} \cdot x)
  \end{equation*}
  is continuous at $b \cdot a$ by continuity of left translations.
  Therefore $f$ is continuous everywhere.
\end{proof}
\begin{theorem}
  Let $G$ be a definable group.  Let $H$ be a definable subgroup.
  Then the inclusion $H \to G$ is a closed embedding with respect to
  the canonical topologies on $H$ and $G$.
\end{theorem}
\begin{proof}
  Let $\tau$ be the canonical topology on $G$.  The restriction $\tau
  \restriction H$ makes $H$ into a definable semimanifold, and $\tau
  \restriction H$ is fixed by left translations of $H$.  By
  Theorem~\ref{thm-improve}, $\tau \restriction H$ is the canonical
  topology on $H$.  So the inclusion $H \to G$ is a topological
  \emph{embedding}.  It remains to show that it is a \emph{closed
  embedding}, i.e., that $H$ is $\tau$-closed.

  Let $\overline{H}$ be the $\tau$-closure of $H$.  Then
  $\overline{H}$ is a $\tau$-closed definable subgroup of $G$.  The
  inclusions $H \hookrightarrow \overline{H} \hookrightarrow G$ are
  embeddings (with respect to the canonical topologies) by the
  previous paragraph.  Restricting to $\overline{H}$, we have a
  definable group $\overline{H}$ with a dense subgroup $H$.  Let $k =
  \dim(\overline{H})$.  Since $\overline{H} \setminus H$ has empty interior in $\overline{H}$, we have $\dim(\overline{H} \setminus H) < k$ by Remark~\ref{interior-duh-2}.  Consequently, $\dim(H) = k$.  But
  $\overline{H} \setminus H$ is a union of cosets of $H$, so it is
  either empty or has dimension at least $k$.  Therefore, it is empty,
  and $H = \overline{H}$.
\end{proof}
\begin{corollary}
  If $f : G \to H$ is a definable injective homomorphism, then $f$ is
  a closed embedding with respect to the canonical topologies on $G$
  and $H$.
\end{corollary}
With a little more work, we get the following:
\begin{corollary}
  Let $G$ be a definable group and $H$ be a definable subgroup.  The
  following are equivalent:
  \begin{enumerate}
  \item $H$ is open as a subset of $G$.
  \item $H$ is clopen as a subset of $G$.
  \item $\dim(H) = \dim(G)$.
  \end{enumerate}
\end{corollary}
\begin{proof}
  (1)$\iff$(2) is clear because $H$ is a closed subgroup, or by the standard proof that open subgroups are clopen subgroups in any topological group.  By
  Lemma~\ref{local-dim} and translation invariance of the topology,
  $\dim(G) = \dim_1(G)$, i.e., $\dim(G)$ equals the local dimension
  $\dim_p(G)$ with $p =1 \in G$.  If $H$ is open, then $\dim(H) =
  \dim_1(H) = \dim_1(G) = \dim(G)$, so (1)$\implies$(3).

  If (3) holds, then every coset of $H$ has the same dimension as $G$.  So if $X \subseteq G$ is invariant under left translation by $G$ (meaning $G \cdot X = X$), then either $X$ is empty or $\dim(X) = \dim(G)$.  Taking $X = \bd(H)$, we get two cases:
  \begin{itemize}
  \item $\dim(\bd(H)) = \dim(G)$, contradicting
    Lemma~\ref{new-tame-sane}(\ref{nts3}).
  \item $\bd(H) = \varnothing$, and then $H$ is clopen.
  \end{itemize}
  Thus (3)$\implies$(2).
\end{proof}
Based on \cite[Corollary~5.22]{admissible} we might expect the following
to hold:
\begin{nontheorem}
  Let $f : G \to H$ be a definable surjective homomorphism.  Then $f$
  is an open map with respect to the canonical topologies on $G$ and
  $H$.
\end{nontheorem}
In fact, this fails already in 2-sorted RCVF $(K,\Gamma)$
\emph{without} an angular component.  In fact, the valuation map $\val
: K \to \Gamma$ fails to be open, since the open set $\{x \in K :
\val(x) = 0\}$ maps to the non-open set $\{0\} \subseteq \Gamma$.
Since we did not include the angular component, this example is
dp-minimal.  We will see in a later paper that the core problem with
this example is the failure of the exchange property (in the 2-sorted
language).

\subsection{Visceral groups}
The next theorem says that if $T$ is a visceral theory (or nice t-minimal theory) expanding the
theory of groups, then we can change the topology witnessing
viscerality to be a group topology.  This clears up some ambiguity in
phrases like ``visceral groups.''
\begin{theorem}
  Suppose the monster model $\Mm$ is a group with respect to
  some definable group operation $\ast$.  Then there is a definable
  group topology $\tau$ such that $\Mm$ is visceral with respect to
  $\tau$.
\end{theorem}
\begin{proof}
  Let $\tau$ be the canonical topology on $(\Mm,\ast)$.  Then $\tau$
  is a definable group topology.  It lifts to a definable uniformity
  in one of the two standard ways.  For example, if $\mathcal{B}$ is a
  definable neighborhood basis of $1 \in \Mm$, then
  \begin{equation*}
    \{E_B : B \in \mathcal{B}\}
  \end{equation*}
  is a definable basis for a uniformity inducing $\tau$, where
  \begin{equation*}
    E_B = \{(x,y) \in \Mm^2 : x \cdot y^{-1} \in B\}.
  \end{equation*}
  It remains to check the following two properties:
  \begin{itemize}
  \item $\tau$ has no isolated points.
  \item Every definable set $D \subseteq \Mm$ has finite boundary with
    respect to $\tau$.
  \end{itemize}
  If some point $a \in \Mm$ is
  $\tau$-isolated, then \emph{every} point is $\tau$-isolated by
  translation-invariance.  By local dimension (Lemma~\ref{local-dim}),
  it follows that $\dim(\Mm) = 0$.  But $\dim(\Mm^n) = n$, so
  $\dim(\Mm^1) = 1$, a contradiction.  Thus there are no isolated
  points.

  Finally, suppose $D \subseteq \Mm$ is definable.  By
  Lemma~\ref{new-tame-sane}(\ref{nts3}), $\dim(\bd(D)) < \dim(\Mm) = 1$, so
  $\bd(D)$ is finite.  Then $\tau$ satisfies all the requirements to
  be a visceral topology.
\end{proof}

\section{Application to definable fields} \label{atdf}
Continue to work in a monster model $\Mm$ of a visceral theory, or
more generally a nice t-minimal theory.
\begin{proposition}
  Let $(K,+,\cdot)$ be a definable field.  Let $\tau$ be the canonical
  topology on the additive group $(K,+)$.
  \begin{itemize}
  \item $\tau$ is a field topology on $K$: the field operations are
    continuous.
  \item $\tau \restriction K^\times$ is the canonical topology on
    $K^\times$.
  \end{itemize}
\end{proposition}
\begin{proof}
  For any $a \in K^\times$, there is a definable homomorphism
  \begin{align*}
    (K,+) &\to (K,+) \\
    x &\mapsto a \cdot x.
  \end{align*}
  By continuity of homomorphisms (Theorem~\ref{continuity}), this map
  is continuous.  It follows that $\tau \restriction K^\times$ is
  translation-invariant.  Then (2) holds by Theorem~\ref{thm-improve}.
  Therefore the maps
  \begin{align*}
    (x,y) &\mapsto xy \\
    x &\mapsto x^{-1}
  \end{align*}
  are continuous, except possibly around 0.  It remains to show that
  multiplication $f(x,y) = xy$ is continuous everywhere.  This
  essentially follows by applying translations.  In more detail, note
  that
  \begin{itemize}
  \item $f$ is continuous in each variable separately (proved above).
  \item $f$ is continuous at $(1,1)$, since $(1,1) \in (K^\times)^2$.
  \end{itemize}
  Let $(a,b) \in K^2$ be
  arbitrary.  The function
  \begin{equation*}
    f(x,y) = (x+a)(y+b) - bx - ay + ab = f(x+a,y+b) - f(b,x) - f(a,y)
    + ab,
  \end{equation*}
  is continuous at $(x,y) = (1-a,1-b)$ since $f$ is continuous at
  $(1,1)$ and continuous in each variable separately (and $+,-$ are
  continuous).  Then $f$ is continuous at any point.
\end{proof}
\begin{example}
  If $T$ is a visceral (or nice t-minimal) theory of fields, we can
  choose the definable topology to be a field topology.
\end{example}
\begin{theorem}
  Let $K$ be an infinite definable field and $\tau$ be
  the canonical field topology on $K$.  Then $\tau$ is st-henselian
  and $K$ is large.
\end{theorem}
\begin{proof}
  Note that
  \begin{itemize}
  \item $\dim(K) > 0$ because $K$ is infinite.
  \item The product topology on $K^n$ agrees with the canonical
    topology on $(K^n,+)$ (Remark~\ref{product-obvious}).
  \item A definable subset $X \subseteq K^n$ has non-empty interior in
    this topology if and only if it has maximum dimension
    (Remark~\ref{interior-duh-2}).
  \end{itemize}
  Then Assumption~\ref{A} holds, with respect to the t-minimal
  dimension theory---each point is either an intrinsic property of
  t-minimal dimension from \cite[\S2.5]{wj-visc-1} or one of the three
  points above.  Assumption~\ref{B}---generic continuity of
  correspondences---holds by Lemma~\ref{new-tame-sane}(\ref{nts1}).
  Theorem~\ref{silver-theorem} then shows that $\tau$ is st-henselian,
  and $K$ is large by Theorem~\ref{bronze-large}.
\end{proof}

% PART 3: the t-minimal case

\section{The t-minimal case: groundwork} \label{ground}

For the remainder of the paper, work in a monster model $\Mm$ of a t-minimal theory.

%% Let $(G,+)$ be a
%% definable abelian group.  We're going to construct a non-trivial
%% definable group topology $\tau$ on $G$ with the property that a set $D
%% \subseteq G$ has non-empty interior if and only if $\dim(D) =
%% \dim(G)$.  The topology is characterized by the fact that
%% \begin{equation*}
%%   \{X-X : X \text{ is a definable subset of $G$ with } \dim(X) = \dim(G)\}
%% \end{equation*}
%% is a neighborhood basis of 0.  In the case of the additive group
%% $(K,+)$ of a definable field $(K,+,\cdot)$, the topology will even be
%% Hausdorff, and will be a field topology on $K$.

%% We'll also prove that generic continuity holds in some form: if $G$
%% and $H$ are two definable abelian groups, and $f : G \to H$ is a
%% definable function, then $f$ is continuous at almost all points, with
%% respect to the new topologies $\tau_G$ and $\tau_H$.

%% Remember that in a t-minimal theory, generic continuity can fail for
%% the original, given topology.  For example, in RCF with the lower
%% limit topology, the function $f(x) = -x$ is nowhere continuous.  On
%% some level, the construction of the new topology $\tau$ is meant to
%% ``repair'' broken examples like the lower limit topology.

\subsection{Good families and good sets} \label{somewhere}
%% A definable family means a family $\{X_b\}_{b \in Y}$ where the sets
%% $Y$ and $\{(a,b) : b \in Y, ~ a \in X_b\}$ are definable.
\begin{definition}
  Let $D$ be a non-empty definable set.
  \begin{enumerate}
  \item A \emph{good family} on $D$ is a definable family
    $\mathcal{F}$ of subsets of $D$ such that
    \begin{itemize}
    \item If $X \in \mathcal{F}$, then $\dim(X) = \dim(D)$.
    \item If $X \subseteq D$ is definable and $\dim(X) = \dim(D)$,
      then there is $X' \in \mathcal{F}$ with $X' \subseteq X$.
    \end{itemize}
  \item $D$ is \emph{good} if it admits a good family.
  \end{enumerate}
\end{definition}
\begin{example}
  If $D \subseteq \Mm^n$ has dimension $n$, let $\mathcal{F}$ be the
  family of non-empty basic open sets $B \subseteq \Mm^n$ with $B
  \subseteq D$.  Then $\mathcal{F}$ is a good family:
  \begin{itemize}
  \item Each non-empty definable open set has dimension $n$.
  \item If $X \subseteq D$ has $\dim(X) = \dim(D) = n$, then $X$ has
    non-empty interior, so it contains a non-empty basic open set.
  \end{itemize}
  Therefore, $D$ is good.
\end{example}
If $a \in \acl(B)$, let $\mult(a/B)$ denote the number of conjugates
of $a$ over $B$, i.e., the size of the orbit $\{a' \in \Mm^n : a'
\equiv_B a\}$.  Note that $1 \le \mult(a/B) < \omega$.
\begin{proposition} \label{good-prop}
  If $X$ is a definable non-empty set, then there is a good definable
  subset $X_0 \subseteq X$ with $\dim(X_0) = \dim(X)$.
\end{proposition}
\begin{proof}
  Let $\mathcal{H}$ be the family of triples $(a,a_0,C)$ where $C$ is
  a finite set defining $X$, $a \in X$, $\dim(a/C) = \dim(X)$, and
  $a_0$ is an acl-basis of $a$ over $C$.  The family $\mathcal{H}$ is
  non-empty (take $C$ defining $X$, take $a \in X$ with $\dim(a/C) =
  \dim(X)$, and take $a_0$ an acl-basis of $a$ over $C$).  Take a
  triple $(a,a_0,C) \in \mathcal{H}$ minimizing $\mult(a/Ca_0)$.  Let
  $\pi$ be the coordinate projection such that $\pi(a) = a_0$.  Let $k
  = \mult(a/Ca_0)$.  Then we can find an $\mathcal{L}(C)$-formula
  $\phi(x,y)$ such that
  \begin{itemize}
  \item $a \in \phi(\Mm,a_0)$.
  \item $|\phi(\Mm,b)| \le k$ for any $b$.
  \end{itemize}
  Let $X_0$ be the set of $a' \in X$ such that $\phi(a',\pi(a'))$
  holds.  Then $a \in X_0$ and $X_0$ is $C$-definable, so $\dim(X_0)
  \ge \dim(a/C) = \dim(X)$ and thus $\dim(X_0) = \dim(X)$.

  Let $Y_0 = \pi(X_0)$.  By choice of $\phi$, the projection $\pi :
  X_0 \to Y_0$ has fibers of size at most $k$.  If $d = \dim(X)$, then
  \begin{equation*}
    \dim(Y_0) = \dim(X_0) = \dim(X) = d = \dim(a/C) = |a_0|.
  \end{equation*}
  Then $\Mm^d$ is the codomain of $\pi$, so $Y_0 \subseteq \Mm^d$ and
  $Y_0$ is a broad subset of $\Mm^d$.  Let \[\mathcal{F} =
  \{\pi^{-1}(B) : B \text{ is a non-empty basic open set and } B
  \subseteq Y_0\},\] where $\pi^{-1}(B)$ is the preimage of $B$ in
  $X_0$.  We claim that $\mathcal{F}$ is a good family on $X_0$:
  \begin{itemize}
  \item If $B$ is a non-empty basic open set in $\Mm^d$ and $B
    \subseteq Y_0$, then $\dim(B) = d$, so the preimage $\pi^{-1}(B)$
    has dimension $d$ also, where $d = \dim(X)$.
  \item Conversely, suppose that $Z \subseteq X_0$ and $\dim(Z) = d$.
    Let $W = X_0 \setminus Z$.  Since $\pi : X_0 \to Y_0$ has finite
    fibers, the image $\pi(Z) \subseteq \Mm^d$ has dimension $d$.
    Then we must be in at least one of the following two cases:
    \begin{description}
    \item[Case I:] $\pi(Z) \setminus \pi(W)$ has dimension $d$.  Then
      there is a non-empty basic open set $B \subseteq \pi(Z)
      \setminus \pi(W) \subseteq \pi(X_0) = Y_0$.  The preimage
      $\pi^{-1}(B) \subseteq X_0$ is contained in $Z$, since $B$ is
      disjoint from $\pi(W)$.  Therefore $Z$ contains a set in
      $\mathcal{F}$.
    \item[Case II:] $\pi(Z) \cap \pi(W)$ has dimension $d$.  Take a
      finite set $C' \supseteq C$ defining $Z$.  The set $\pi(Z) \cap
      \pi(W)$ is a $C'$-definable broad subset of $\Mm^d$, so there is
      some $a'_0 \in \pi(Z) \cap \pi(W)$ such that $\dim(a'_0/C') =
      d$, or equivalently, $a'_0$ is acl-independent over $C'$.  Take
      $a' \in Z$ with $\pi(a') = a'_0$.  Since $\pi : X_0 \to Y_0$ has
      finite fibers, $a'$ and $a'_0$ are interalgebraic over $C$ and
      $C'$.  Then $a'_0$ is an acl-basis for $a'$ over $C'$.  Also,
      $\dim(a'/C) = \dim(a'_0/C) = d = \dim(X)$.  Therefore
      $(a',a'_0,C') \in \mathcal{H}$.  By choice of $(a,a_0,C)$, we
      have
      \begin{equation*}
        \mult(a'/C'a'_0) \ge \mult(a/Ca_0) \ge k,
      \end{equation*}
      where $k$ bounds the sizes of the fibers of $\pi : X_0 \to Y_0$.
      Let $b_1,\ldots,b_k$ be distinct conjugates of $a'$ over
      $C'a'_0$.  For each $i$, $b_i \in Z$ because $a' \in Z$ and $Z$
      is $C'$-definable, and $\pi(b_i) = a'_0$ because $\pi(a') =
      a'_0$.  Then $\{b_1,\ldots,b_k\}$ is in the fiber of $\pi : X_0
      \to Y_0$ over $a'_0$.  Since this fiber has size at most $k$,
      the set $\{b_1,\ldots,b_k\}$ must be the entire fiber.  But
      $\{b_1,\ldots,b_k\}$ is disjoint from $W$, contradicting the
      fact that $a'_0 \in \pi(W)$.  This case cannot happen.  \qedhere
    \end{description}
  \end{itemize}
\end{proof}

\subsection{Miscellaneous facts}
\begin{fact}[Lemma on products] \label{prod-fact}
  Let $X,Y,D$ be definable sets with
  \begin{gather*}
    D \subseteq X \times Y \\
    \dim(D) = \dim(X \times Y).
  \end{gather*}
  Then there are definable sets $X_0 \subseteq X$ and $Y_0 \subseteq Y$ with
  \begin{gather*}
    X_0 \times Y_0 \subseteq D \\
    \dim(X_0 \times Y_0) = \dim(D) = \dim(X \times Y).
  \end{gather*}
\end{fact}
If $X$ and $Y$ were open subsets of $\Mm^n$ and $\Mm^m$, this would be relatively trivial.  For the general case, we reduce to the open case by doing something with multiplicities $\mult(a/Ca_0)$ similar to the proof of Proposition~\ref{good-prop}.  However, the proof is much more
complicated, and we postpone it to Appendix~\ref{app-B}.
\begin{proposition} \label{splitting}
  If $X \subseteq \Mm^n$ is definable and infinite, then there are
  disjoint definable subsets $X_0,X_1 \subseteq X$ with $\dim(X_0) =
  \dim(X_1) = \dim(X)$.
\end{proposition}
\begin{proof}
  Recall from \cite[Definition~2.47]{wj-visc-1} that a set $C
  \subseteq \Mm^n$ is a \emph{weak cell} if there is a coordinate
  projection $\pi : \Mm^n \to \Mm^k$ such that $\pi(C)$ is a non-empty
  open set, and the map $C \to \pi(C)$ has finite fibers.  Every
  definable set can be written as a finite disjoint union of weak
  cells \cite[Proposition~2.50]{wj-visc-1}.  Decompose the given
  set $X$ into weak cells $X = C_1 \sqcup \cdots \sqcup C_m$.  Some
  $C_i$ has the same dimension as $X$.  Replacing $X$ with $C_i$, we
  may assume $X$ is a weak cell.  Let $\pi : X \to \pi(X)$ be the
  finite-to-one projection onto an open subset of $\Mm^d$, where $d =
  \dim(X)$.  Since $X$ is infinite, $d > 0$.  Then by Hausdorffness we
  can find two disjoint non-empty definable open sets $B_0, B_1
  \subseteq \pi(X)$.  Let $X_i$ be the preimage of $B_i$ in $X$.
  Since $X \to \pi(X)$ is a finite-to-one surjection, $\dim(X_i) =
  \dim(B_i) = \dim(\Mm^d) = d$.  The preimages $X_0$ and $X_1$ are
  disjoint because $B_0$ and $B_1$ are.
\end{proof}
\begin{remark} \label{subset}
  If $X \subseteq \Mm^n$ is definable and $d \le \dim(X)$ then there
  is a definable $X' \subseteq X$ with $\dim(X') = d$.
\end{remark}
\begin{proof}
  The proof is similar to Proposition~\ref{splitting}.  We can assume $X$ is a weak
  cell with a finite-to-one projection $\pi : X \to \pi(X)$, where
  $\pi(X) \subseteq \Mm^m$ is open and $m = \dim(X)$.  Since $\pi(X)$
  is open and $d \le m$, we can easily find a definable subset $D
  \subseteq \pi(X)$ with $\dim(D) = d$.  (For example, $\pi(X)$
  contains a product $\prod_{i=1}^m B_i$ where each $B_i$ is a
  non-empty basic open set.  Take $a_i \in B_i$ for each $i$.  Then we
  can take $D = \prod_{i=1}^d B_i \times \prod_{i=d+1}^m \{a_i\}$.)
  Finally, take $X' = \pi^{-1}(D)$.  Then $\dim(X') = \dim(D) = d$.
\end{proof}

\section{The canonical topology on definable abelian groups}
Fix a definable abelian group $(G,+)$ of dimension $d$.  A
definable set $X \subseteq G$ is \emph{big} if $\dim(X) = d$.  If $X,Y
\subseteq G$, then $X-Y$ denotes $\{x-y : x \in X, y \in Y\}$.
\begin{question}
  Can we generalize the arguments in this section to deal with
  non-abelian definable groups?
\end{question}
\subsection{Basic neighborhoods}
\begin{definition}
  A \emph{basic neighborhood (of 0)} is a set of the form $X - X$,
  where $X \subseteq G$ is big.
\end{definition}
\begin{remark}\label{duh}
Any basic neighborhood $X-X$ is big, because it contains $b-X$ for
any $b \in X$, and $\dim(b-X) = \dim(X) = d$.  
\end{remark}
\begin{lemma} \label{slide}
  If $X,Y \subseteq G$ are big, then there is $\delta \in G$ such that
  $X \cap (Y + \delta)$ is big.
\end{lemma}
\begin{proof}
  Take a small set $C$ defining $X$ and $Y$.  Take $(a,b) \in X \times
  Y$ with \[\dim(a,b/C) = \dim(X \times Y) = \dim(X) + \dim(Y) = 2d.\]
  Let $\delta = a - b$.  Since $(a,b)$ and $(a,\delta)$ are
  interdefinable over $C$, we have $\dim(a,\delta/C) = \dim(a,b/C) =
  2d$.  By subadditivity of dimension
  \cite[Proposition~2.31(5)]{wj-visc-1},
  \begin{equation*}
    2d = \dim(a,\delta/C) \le \dim(a/C\delta) + \dim(\delta/C) \le 2d,
  \end{equation*}
  where the inequality on the right holds because $a$ and $\delta$ are
  elements of $G$ and $\dim(G) = d$.  Then equality must hold, and
  $\dim(a/C\delta) = d$.  The element $a = b+\delta$ belongs to the
  $C\delta$-definable set $X \cap (Y + \delta)$, so this set must have
  dimension at least $d$.
\end{proof}
\begin{proposition} \label{filtered}
  If $U$ and $V$ are basic neighborhoods, there is a
  basic neighborhood $W$ such that $W \subseteq U \cap V$.
\end{proposition}
\begin{proof}
  Write $U$ and $V$ as $X-X$ and $Y-Y$.  Replacing $Y$ with a
  translate $Y+\delta$ (which doesn't affect $Y-Y$), we may assume the
  set $Z = X \cap Y$ is big.  Then $Z-Z \subseteq (X-X) \cap (Y-Y)$.
\end{proof}
\begin{lemma} \label{later-interior}
  If $D \subseteq G$ is big, then $D$ contains $X-X + \delta$ for some
  big set $X$ and element $\delta \in G$.
\end{lemma}
\begin{proof}
  Let $S = \{(x,y) \in G^2 : x-y \in D\}$.  Then $S$ is in definable
  bijection with $G \times D$ via the map
  \begin{align*}
    S &\to G \times D \\
    (x,y) &\mapsto (x,x-y).
  \end{align*}
  In particular, $\dim(S) = \dim(G \times D) = 2d$.  By the Lemma on
  Products (Fact~\ref{prod-fact}), the set $S \subseteq G \times G$
  contains a product $U \times V$ of two big sets.  In particular, $U
  - V \subseteq D$.  By Lemma~\ref{slide}, there is some $\delta$ such
  that $X = U \cap (V + \delta)$ is big.  Then $X \subseteq U$ and $X
  - \delta \subseteq V$ and so
  \begin{equation*}
    X - X + \delta = X - (X - \delta) \subseteq U - V \subseteq D.  \qedhere
  \end{equation*}
\end{proof}
\begin{proposition} \label{subtraction}
  If $U$ is a basic neighborhood, then there is a basic neighborhood
  $V$ such that $V - V \subseteq U$.
\end{proposition}
\begin{proof}
  Write $U$ as $D - D$ with $D$ big.  Lemma~\ref{later-interior} gives
  a basic neighborhood $V = X-X$ and element $\delta \in G$ with $V +
  \delta \subseteq D$.  Then
  \begin{equation*}
    U = D - D \supseteq (V + \delta) - (V + \delta) = V - V. \qedhere
  \end{equation*}
\end{proof}
\subsection{The topology}
\begin{theorem} \label{cantop}
  There is a group topology $\tau$ on $(G,+)$ characterized by the
  fact that the basic neighborhoods form a neighborhood basis of 0.
\end{theorem}
(Note that we are not saying that $\tau$ is definable or Hausdorff.)
\begin{proof}
  This follows formally from Propositions~\ref{filtered} and
  \ref{subtraction}, together with the trivial fact that $0 \in U$ for
  any basic neighborhood $U$.
\end{proof}
\begin{definition}
  The \emph{canonical topology} on a definable abelian group $(G,+)$
  is the topology from Theorem~\ref{cantop}.
\end{definition}
\begin{theorem} \label{definable}
  The topology $\tau$ is definable: there is a definable family
  $\mathcal{B}$ which is a basis of open sets.
\end{theorem}
\begin{proof}
  By Proposition~\ref{good-prop}, there is a big subset $G_0 \subseteq
  G$ admitting a good family $\mathcal{F}$.  Each $X \in \mathcal{F}$
  has $\dim(X) = \dim(G_0) = \dim(G)$, so each $X \in \mathcal{F}$ is
  big.
  \begin{claim} \label{huati}
    If $Y \subseteq G$ is big, there is $X \in \mathcal{F}$ and
    $\delta \in G$ with $X + \delta \subseteq Y$.
  \end{claim}
  \begin{claimproof}
    Lemma~\ref{slide} gives some $\delta$ such that $G_0 \cap (Y -
    \delta)$ is big.  Since $\mathcal{F}$ is a good family, there is
    some $X \in \mathcal{F}$ with $X \subseteq G_0 \cap (Y - \delta)$.
    Then $X + \delta \subseteq Y$.
  \end{claimproof}
  Let $\mathcal{B}_0 = \{X-X : X \in \mathcal{F}\}$.  Then
  $\mathcal{B}_0$ is a definable family of basic neighborhoods.  We
  claim that $\mathcal{B}_0$ is a neighborhood basis of 0.  Given a
  basic neighborhood $U$, we can write $U$ as $Y-Y$ for some big set
  $Y$.  Claim~\ref{huati} gives $X \in \mathcal{F}$ and $\delta \in G$
  with $X + \delta \subseteq Y$.  Then
  \begin{gather*}
    \mathcal{B}_0 \ni X - X \\
    X - X = (X + \delta) - (X + \delta) \subseteq Y
    - Y = U.
  \end{gather*}
  Thus $\mathcal{B}_0$ is a definable neighborhood basis of 0.  It
  follows formally that the topology $\tau$ is definable.
%%   In more detail,
%%   \begin{itemize}
%%   \item First use the existence of a definable neighborhood basis to
%%     show that interiors are uniformly definable: if $\{X_b\}_{b \in
%%       Y}$ is a definable family then so is $\{\ter(X_b)\}_{b \in Y}$,
%%     since $a \in \ter(X_b)$ if and only if there is $U \in
%%     \mathcal{B}_0$ such that $a + U \subseteq X_b$.
%%   \item Then observe that the definable family
%%   \begin{equation*}
%%     \{\ter(U) + \delta : U \in \mathcal{B}_0, ~ \delta \in G\}
%%   \end{equation*}
%%   is a basis of open sets.  \qedhere
%%   \end{itemize}
\end{proof}
\subsection{A property like viscerality}
\begin{theorem} \label{quasi}
  If $D \subseteq G$ is definable, then the ($\tau$-)interior
  $\ter(D)$ is non-empty if and only if $D$ is big.
\end{theorem}
\begin{proof}
  If $a \in \ter(D)$, then $D$ contains $a+U$ for some basic
  neighborhood $U$.  The set $U$ is big, so $a+U$ and $D$ are big too.

  Conversely, suppose $D$ is big.  Lemma~\ref{later-interior} shows
  that $D$ contains $U + \delta$ for some basic neighborhood $U = X-X$
  and some element $\delta \in G$.  Then $\delta \in \ter(D)$.
\end{proof}
\begin{corollary}
  If $D \subseteq G$ is definable, then the ($\tau$-)boundary $\bd(D)$
  isn't big.
\end{corollary}
\begin{proof}
  Since the topology is definable, the closure $\cl(D)$ and interior
  $\ter(D)$ are definable.  The boundary $\bd(D)$ is the union of the
  two sets $\cl(D) \setminus D$ and $D \setminus \ter(D)$.  Both sets
  have empty interior, so neither is big, and neither is their union.
\end{proof}
\begin{corollary} \label{big-inter}
  If $D \subseteq G$ is big, then the interior $\ter(D)$ is big.
\end{corollary}
\begin{proof}
  $D \setminus \ter(D)$ isn't big, so $\ter(D)$ must be big.
\end{proof}
\begin{remark} \label{upgrade}
  If $\dim(G) = 1$, then Theorem~\ref{quasi} really implies that $G$
  is visceral with respect to $\tau$, and so \cite{visceral,wj-visc-1}
  immediately give all the desired ``tame topology'' theorems on $G$.
  For example, definable functions $f : G^n \to G^m$ are generically
  continuous.  When $\dim(G) > 1$, this will take more work.
\end{remark}
\begin{lemma} \label{characterization}
  $\tau$ is the unique definable group topology on $G$ with the
  property that for any definable set $D \subseteq G$, \[\dim(D) =
  \dim(G) \iff \ter_\tau(D) \ne \varnothing.\]
\end{lemma}
\begin{proof}
  Let $\tau'$ be another group topology with this property.  It
  suffices to prove the two claims below:
  \begin{claim}
    If $U$ is a $\tau$-neighborhood of 0, then $U$ is a
    $\tau'$-neighborhood of 0.
  \end{claim}
  To prove this, take a big definable set $X \subseteq G$ with $U
  \supseteq X-X$.  By the assumption on $\tau'$, there is some $b \in
  X$ with $b \in \ter_{\tau'}(X)$.  Then $U \supseteq X - X \supseteq
  X-b$, and $X-b$ is a $\tau'$-neighborhood of 0.
  \begin{claim}
    If $U$ is a $\tau'$-neighborhood of 0, then $U$ is a
    $\tau$-neighborhood of 0.
  \end{claim}
  To prove this, take a definable $\tau'$-neighborhood $V \ni 0$ with
  $V-V \subseteq U$.  Then $0 \in \ter_{\tau'}(V) \ne \varnothing$, so
  $V$ is big, and $V-V$ is a $\tau$-neighborhood of 0.
\end{proof}
% Note to self: it's not enough to assume that $\tau$ is a monoid topology, because of the lower limit topology on RCF.
\begin{remark} \label{2-approaches-agree}
  If the theory is visceral or nicely t-minimal, then the canonical
  topology we constructed earlier (Definition~\ref{cantop0}) satisfies
  the condition that a definable set $X \subseteq G$ has non-empty
  interior iff $\dim(X) = \dim(G)$ (see Remark~\ref{interior-duh-2}).
  By Lemma~\ref{characterization}, it agrees with the new canonical
  topology constructed here.
\end{remark}

\subsection{Discreteness and triviality?}
\begin{theorem}
  If $\dim(G) > 0$, then the topology $\tau$ is non-discrete, and has
  no isolated points.
\end{theorem}
\begin{proof}
  Every basic neighborhood is big, hence infinite, so $0$ is not an
  isolated point.  Since the translation $f(x) = x+a$ is a
  homeomorphism, no point $a \in G$ is isolated.
\end{proof}
\begin{theorem} \label{non-triv}
  If $G$ is non-trivial (meaning $|G| > 1$), then the topology $\tau$
  is non-trivial---there are open sets other than $G$ and $\{0\}$.
\end{theorem}
\begin{proof}
  First suppose $G$ is finite.  Then any non-empty subset of $G$ is
  big, and it is easy to see that the topology $\tau$ is discrete.  As
  long as $|G| > 1$, the discrete topology is non-trivial.

  Next suppose $G$ is infinite.  Proposition~\ref{splitting} gives two
  disjoint definable big sets $X_0, X_1 \subseteq G$.
  Lemma~\ref{slide} gives $\delta$ such that the intersection $Z = X_0
  \cap (X_1 + \delta)$ is big.  Then $Z + \delta \subseteq X_0 +
  \delta$ and $Z \subseteq X_1 + \delta$, so $Z \cap (Z + \delta) =
  \varnothing$.  Equivalently, $\delta \notin Z - Z$.  Then
  \emph{some} basic neighborhood $Z-Z$ is strictly smaller than $G$,
  which prevents the topology from being trivial.
\end{proof}
\begin{question} \label{hq}
  Is $\tau$ always Hausdorff?
\end{question}
\begin{remark} \label{haus-rem}
  Let $G_0$ be the intersection of all the basic neighborhoods.  Using
  Proposition~\ref{subtraction} and Theorem~\ref{definable}, one can
  show that $G_0$ is a definable subgroup of $G$.  Question~\ref{hq}
  equivalently asks whether $G_0$ is trivial.  We can at least show
  that $\dim(G_0) < \dim(G)$ as follows.  Theorem~\ref{quasi} implies that any
  big definable set contains a coset of $G_0$.  If $G_0$ itself is
  big, then $G_0 \setminus \{1\}$ is big, but fails to contain a coset
  of $G_0$, a contradiction.
\end{remark}

\subsection{The canonical topology on products}
Let $G$ and $H$ be two abelian definable groups.
\begin{theorem} \label{product-top}
  The canonical topology $\tau_{G \times H}$ on $G \times H$ is the
  product topology of $\tau_G$ and $\tau_H$.
\end{theorem}
\begin{proof}
  Both topologies are group topologies, so it suffices to show that
  they have the same neighborhoods of 0.  In the product topology, a
  basic neighborhood of 0 has the form $U \times V$ where $U$ and $V$
  are basic neighborhoods in $G$ and $H$, respectively.  Write $U$ as
  $X-X$ and $V$ as $Y-Y$ for some big definable subsets $X \subseteq
  G$ and $Y \subseteq H$.  Then $X \times Y$ is a big definable subset
  of $G \times H$, and
  \begin{equation*}
    (X \times Y) - (X \times Y) = \{(x_1-x_2,y_1-y_2) : x_1,x_2 \in X, y_1,y_2 \in Y\} =  (X - X) \times (Y - Y) = U \times V.
  \end{equation*}
  So $U \times V$ is already a basic neighborhood in the canonical
  topology on $G \times H$.

  Conversely, suppose $W$ is a basic neighborhood in the canonical
  topology on $G \times H$.  Then $W = Z - Z$ for some big $Z
  \subseteq G \times H$.  By the Lemma on Products
  (Fact~\ref{prod-fact}), $Z$ contains $X \times Y$ for some big sets
  $X \subseteq G$ and $Y \subseteq H$.  Then $W$ contains $(X \times
  Y) - (X \times Y) = (X-X) \times (Y-Y)$, which is a basic
  neighborhood in the product topology.
\end{proof}

\section{A partial form of generic continuity} \label{apof}
Continue to fix a definable abelian group $(G,+)$.
\subsection{Hammer lemma}
Our only tool is the following hammer.
\begin{lemma}[Hammer lemma] \label{hammer}
  Let $\mathcal{B}_0$ be a definable neighborhood basis of 0 in
  $(G,+)$.  Let $X$ be a definable set.  Let $\bowtie$ be a relation
  between $X$ and $\mathcal{B}_0$, with $a \bowtie N$ pronounced as
  ``$a$ is compatible with $N$''.  Suppose $\bowtie$ satisfies the
  following properties:
  \begin{enumerate}
  \item For any $a \in X$, there is $N \in \mathcal{B}_0$ such that $a
    \bowtie N$.
  \item If $a \bowtie N$ and $N \supseteq N' \in \mathcal{B}_0$, then
    $a \bowtie N'$.
  \item The relation $\bowtie$ is definable, i.e., the set $\{(a,N) :
    a \bowtie N\}$ is a definable subset of $X \times \mathcal{B}_0$.
  \end{enumerate}
  Then there is a definable subset $X' \subseteq X$ with $\dim(X') =
  \dim(X)$ and a neighborhood $N \in \mathcal{B}_0$ such that for any
  $a \in X'$, $a \bowtie N$.
\end{lemma}
\begin{proof}
  Like the proof of \cite[Lemma~3.1]{wj-visc-1}.  In more detail,
  let $X_N = \{a \in X : a \bowtie N\}$ for $N \in \mathcal{B}_0$.
  Then $\{X_N\}_{N \in \mathcal{B}_0}$ is a definable family by
  condition (3).  Condition (2) says that if $N' \subseteq N$ then
  $X_{N'} \supseteq X_N$.  Since the family $\mathcal{B}_0$ is
  downward directed, the family $\{X_N : N \in \mathcal{B}_0\}$ is
  upward directed.  By condition (1), $X = \bigcup_{N \in
    \mathcal{B}_0} X_N$.  The union is filtered, so
  \cite[Corollary~2.43]{wj-visc-1} gives some $N$ such that
  $\dim(X_N) = \dim(X)$.  Take $X' = X_N$.
\end{proof}

\subsection{Local dimension}
If $X \subseteq G$ is a definable set and $p \in X$, let $\dim_p(X)$
denote the local dimension of $X$ at $p$, i.e., $\dim(X \cap N)$ for
any sufficiently small definable neighborhood $N$ of $p$ in the
canonical topology.  Local dimension depends definably on $X$ and $p$,
because dimension is definable \cite[Theorem~2.51]{wj-visc-1} and
the canonical topology is definable.
\begin{lemma}
  If $X \subseteq G$ is definable and $\dim_p(X) \le k$ for every $p
  \in X$, then $\dim(X) \le k$.
\end{lemma}
\begin{proof}
  Fix a definable neighborhood basis $\mathcal{B}_0$ of $0$ in $G$.
  For $a \in X$ and $N \in \mathcal{B}_0$, say that $a \bowtie N$ if
  $\dim((a + N) \cap X) \le k$.  Then $\bowtie$ satisfies the
  conditions of the hammer lemma (Lemma~\ref{hammer}), so there is
  some definable set $X_1 \subseteq X$ with $\dim(X_1) = \dim(X)$ and
  some neighborhood $N_1 \in \mathcal{B}_0$ such that for any $a \in
  X_1$, we have $a \bowtie N_1$, i.e., $\dim((a + N_1) \cap X) \le k$.
  Since $X_1 \subseteq X$, we also have
  \begin{equation*}
    \dim((a + N_1) \cap X_1) \le k \text{ for } a \in X_1.
  \end{equation*}
  Take $N_2 \in \mathcal{B}_0$ small enough that $N_2 - N_2 \subseteq
  N_1$.
  \begin{claim} \label{claim-a}
    If $a \in G$, then $\dim((a + N_2) \cap X_1) \le k$.
  \end{claim}
  \begin{claimproof}
    If $(a + N_2) \cap X_1 = \varnothing$, then the dimension is
    $-\infty$, which is at most $k$.  Otherwise, take some $a' \in (a
    + N_2) \cap X_1$.  Then $a \in a' - N_2$, so
    \begin{equation*}
      a + N_2 \subseteq (a' - N_2) + N_2 \subseteq a' + N_1.
    \end{equation*}
    Then $(a + N_2) \cap X_1$ is a subset of $(a' + N_1) \cap X_2$,
    which has dimension at most $k$.
  \end{claimproof}
  Consider the map
  \begin{gather*}
    f : X_1 \times N_2 \to G \\
    f(x,y) = x-y.
  \end{gather*}
  \begin{claim} \label{claim-b}
    The fibers of $f$ have dimension at most $k$.
  \end{claim}
  \begin{claimproof}
    Fix $a \in G$.  The fiber $f^{-1}(a)$ is the set
    \begin{equation*}
      \{(x,y) \in X_1 \times N_2 : x-y = a\} =
      \{(x,x-a) : x \in X_1 \cap (a + N_2)\}.
    \end{equation*}
    By Claim~\ref{claim-a}, this set has dimension at most $k$.
  \end{claimproof}
  By
  \cite[Theorem~2.37(8)]{wj-visc-1}, Claim~\ref{claim-b} implies
  \begin{equation*}
    \dim(X_1) + \dim(N_2) \le k + \dim(G).  \tag{$\ast$}
  \end{equation*}
  Let $d = \dim(G)$.  Every neighborhood of 0 has dimension
  $d$ (Remark~\ref{duh}), so $\dim(N_2) = d$.  Subtracting $d$ from
  both sides of ($\ast$) gives $\dim(X_1) \le k$.  But $X_1$ was
  chosen so that $\dim(X) = \dim(X_1)$.
\end{proof}
\begin{corollary} \label{local-cor}
  $\dim(X) = \max_{p \in X} \dim_p(X)$ for definable $X \subseteq G$.
\end{corollary}
\begin{example}
  If $X \subseteq G$ is definable and discrete, then $\dim_p(X) = 0$
  for any $p \in X$, so $\dim(X) = 0$ and $X$ is finite.
\end{example}
We can also stratify definable sets by local dimension:
\begin{theorem} \label{stratify}
  Let $X$ be a non-empty definable subset of $G$.  For $k \le
  \dim(X)$, let $X_{(k)}$ be the set of $p \in X$ with $\dim_p(X) = k$.
  \begin{enumerate}
  \item $\dim_p(X_{(k)}) = k$ for each $p \in X_{(k)}$.  In
    particular, $\dim(X_{(k)}) = k$ or $X_{(k)}$ is empty.
  \item If $k = \dim(X)$, then $X_{(k)} \ne \varnothing$.
  \end{enumerate}
\end{theorem}
\begin{proof}
  The second point is immediate from Corollary~\ref{local-cor}.
  Let $X_{(<k)} = \bigcup_{i=0}^{k-1} X_{(i)}$ and $X_{(>k)} =
  \bigcup_{i=k+1}^\infty X_{(i)}$.  First note that if $p \in
  X_{(k)}$, then $ \dim_p(X_{(k)}) \le \dim_p(X) = k$, so
  $\dim(X_{(k)}) \le k$ by Corollary~\ref{local-cor}.  Consequently,
  $\dim(X_{(<k)}) < k$.
  %% If $k = \dim(X)$, then $\dim(X_{(<k)}) < k =
%%   \dim(X)$, so the complement $X \setminus X_{(<k)} = X_{(k)}$ is
%%   non-empty, proving the second point.

  Given $p \in X_{(k)}$, it remains to show that $\dim_p(X_{(k)}) \ge
  k$.  Suppose that $\dim_p(X_{(k)}) < k$ for the sake of
  contradiction.  Then there is an open neighborhood $N \ni p$ with
  $\dim(X_{(k)} \cap N) < k$.  Taking $N$ small enough, we may assume
  $\dim(X \cap N) = k$.  If $q \in X \cap N$, then the neighborhood $N
  \ni q$ shows that $\dim_q(X) \le k$, so $q \notin X_{(>k)}$.  Then
  $X \cap N$ splits as
  \begin{equation*}
    X \cap N = (X_{(k)} \cap N) \sqcup (X_{(<k)} \cap N).
  \end{equation*}
  We assumed $\dim(X_{(k)} \cap N) < k$, and we showed earlier that
  $\dim(X_{(<k)}) < k$.  Then $\dim(X \cap N) < k$, a contradiction.
  This proves the first point.
\end{proof}
\begin{corollary} \label{stratify-cor}
  If $X \subseteq G$ is definable and non-empty, there is a non-empty
  definable $X' \subseteq X$ with $\dim_p(X') = \dim(X') = \dim(X)$
  for any $p \in X'$.
\end{corollary}
\begin{proof}
  In the notation of Theorem~\ref{stratify}, take $X' = X_{(k)}$ with
  $k = \dim(X)$.
\end{proof}
\begin{example}
  For any $0 \le d \le \dim(G)$, there is a non-empty definable set $X
  \subseteq G$ such that $\dim(X) = d$ and $\dim_p(X) = d$ for any $p
  \in X$.
\end{example}
\begin{proof}
  By Remark~\ref{subset}, there is a definable set $Y \subseteq G$
  with $\dim(Y) = d$.  Applying Corollary~\ref{stratify-cor} we get a
  subset $X \subseteq Y$ with $\dim_p(X) = d$ for any $p \in X$.
\end{proof}

\subsection{Generic continuity}
\begin{theorem} \label{gencon}
  Let $(G,+)$ and $(H,+)$ be two definable abelian groups.  Let $X
  \subseteq G$ be a big definable subset, meaning $\dim(X) = \dim(G)$.
  Let $f : X \to H$ be a definable function.  Then
  \begin{equation*}
    \dim \{p \in X : f \text{ isn't continuous at } p\} < \dim(G).
  \end{equation*}
\end{theorem}
The proof is based on the proof of \cite[Proposition~3.12]{visceral}.
\begin{proof}
  Let $X'$ be the set of points of discontinuity.  Let $X''$ be the
  interior of $X'$.  By Theorem~\ref{quasi}, it suffices to show that
  $X''$ is empty.  Otherwise, we have a non-empty, open, big
  definable subset $X'' \subseteq G$ and a definable function $f : X''
  \to H$ which is nowhere continuous.  Replacing $X$ with $X''$, we
  may assume that $f : X \to H$ is nowhere continuous.

  Let $\mathcal{B}^G_0$ and $\mathcal{B}^H_0$ be definable
  neighborhood bases of 0 in the two groups $G$ and $H$.  For $a \in
  X$ and $N \in \mathcal{B}^H_0$, let $a \bowtie N$ mean that for any
  $N' \in \mathcal{B}^G_0$, there is $x \in a + N'$ with $f(x) \notin
  f(a) + N$.  For any $a \in X$, the discontinuity of $f$ at $a$ means that $a
  \bowtie N$ for some $N \in \mathcal{B}^H_0$.  The relation $\bowtie$
  then satisfies the conditions of the hammer lemma
  (Lemma~\ref{hammer}), so there is some $N_0 \in \mathcal{B}^H_0$ and
  some big $X' \subseteq X$ such that for any $a \in X'$, we have $a
  \bowtie N_0$.  Take a smaller $N_1 \in \mathcal{B}^H_0$ with $N_1 -
  N_1 \subseteq N_0$.

  Let $d = \dim(G) = \dim(X) = \dim(X')$.  Let $\Gamma$ be the graph
  of $f \restriction X'$:
  \begin{equation*}
    \Gamma = \{(x,f(x)) : x \in X'\}.
  \end{equation*}
  Then $\dim(\Gamma) = \dim(X') = d$.  By Corollary~\ref{local-cor},
  there is some $(a,f(a)) \in \Gamma$ such that the local dimension of
  $\Gamma$ at $(a,f(a))$ is $d$.  Note that $G \times (f(a) + N_1)$ is
  a neighborhood of $(a,f(a))$ by Theorem~\ref{product-top},
  so \[\dim((G \times (f(a) + N_1)) \cap \Gamma) = d.\] Let $X''
  \subseteq X'$ be the projection of $(G \times (f(a) + N_1)) \cap
  \Gamma$ onto the first coordinate.  The projection is injective on
  $\Gamma$, so $\dim(X'') = d$ and $X''$ has non-empty interior
  (Theorem~\ref{quasi}).  Take a point $a' \in \ter(X'') \subseteq X''
  \subseteq X'$.  Because $a' \in X'$, we have $a' \bowtie N_0$.  As
  $X''$ is a neighborhood of $a'$, there is $a'' \in X'$ such that
  $f(a'') \notin f(a') + N_0$.  That is
  \begin{equation*}
    f(a'') - f(a') \notin N_0.
  \end{equation*}
  On the other hand, $a', a'' \in X''$, so $f(a'), f(a'') \in f(a) +
  N_1$.  Then $f(a'') - f(a') \in N_1 - N_1 \subseteq N_0$, a
  contradiction.
\end{proof}
\begin{corollary} \label{homcor}
  If $f : G \to H$ is a definable homomorphism, then $f$ is
  continuous.
\end{corollary}
\begin{proof}
  Theorem~\ref{gencon} gives \emph{some} point of continuity $a_0 \in
  G$.  Applying translations, we see that $f$ is continuous at
  \emph{any} point $a \in G$.  Indeed, if $t_1 : G \to G$ is the
  translation moving $a$ to $a_0$ and $t_2 : H \to H$ is the
  translation moving $f(a_0)$ to $f(a)$, then the composition $t_2
  \circ f \circ t_1$ is continuous at $a$.  But this composition is
  exactly $f$:
  \begin{gather*}
    t_2(f(t_1(x))) = t_2(f(x + (a_0 - a))) = t_2(f(x) + f(a_0) - f(a)) \\
    = (f(x) + f(a_0) - f(a)) + (f(a) - f(a_0)) = f(x).  \qedhere
  \end{gather*}
\end{proof}
\begin{remark} \label{isorem}
  If $f : G \to H$ is a definable \emph{isomorphism}, then continuity
  of $f$ can be seen much more directly: because the canonical
  topology is defined in a canonical way, the two topologies $\tau_G$
  and $\tau_H$ must correspond under the isomorphism $f$.
\end{remark}
\begin{question}
  If $f : G \to H$ is an injective definable homomorphism, is $f$
  necessarily a closed embedding?
\end{question}
\begin{remark}
  This question is connected to Question~\ref{hq}, since the case where
$G$ is trivial amounts to the question of whether $\{0_H\}$ is closed
in $H$, which is equivalent to asking whether $H$ is Hausdorff.  More
generally, if $G$ is a definable abelian group and $G_0$ is the group
from Remark~\ref{haus-rem}, then the subspace topology on $G_0$ is
trivial, so the inclusion $G_0 \to G$ could not be an embedding unless $G_0$ is a
trivial group, by Theorem~\ref{non-triv}.
\end{remark}
\begin{question}
  In visceral theories, we can prove other stronger forms of generic
  continuity, such as the following:
  \begin{itemize}
  \item If $f : X \to Y$ is definable, then we can partition $X$ into
    finitely many definable subsets on which $f$ is continuous
    \cite[Theorem~3.14(1)]{wj-visc-1}.
  \item If $f : X \to Y$ is definable and $X$ is non-empty, then $f$
    is continuous on a dense subset of $X$
    \cite[Theorem~4.13]{wj-visc-1}.
  \end{itemize}
  Could we prove similar results in this setting?
\end{question}
\begin{remark}
  The proof of Theorem~\ref{gencon} \emph{almost} proves the following:
  \begin{nonclaim} \label{nc}
    If $X \subseteq G$ is definable and $f : X \to H$ is definable,
    then $\{p \in X : f \text{ is discontinuous at } p\}$ has lower
    dimension than $X$.
  \end{nonclaim}
  Unfortunately, the proof fails near the end, when we need a big
  subset of $X$ to have non-empty interior (i.e., when choosing the
  element $a'$).  The analogue of Non-Claim~\ref{nc} in visceral
  theories fails \cite[Example~4.5]{wj-visc-1}, so it is probably
  too optimistic to hope for Non-Claim~\ref{nc} to hold in the current
  setting.
\end{remark}
In order to verify Assumption~\ref{B} for definable fields, we will
also need generic continuity of correspondences:
\begin{theorem}\label{gencon2}
  Let $(G,+)$ and $(H,+)$ be two definable abelian groups.  Suppose
  the canonical topology on $H$ is Hausdorff.  Let $X \subseteq G$ be
  a big definable set.  Let $f : X \rightrightarrows H$ be a definable
  $k$-correspondence.  Then
  \begin{equation*}
    \dim \{p \in X : f \text{ isn't continuous at } p \} < \dim(G).
  \end{equation*}
\end{theorem}
\begin{proof}
  Proceed by induction on $k$.  The base case $k=1$ is
  Theorem~\ref{gencon}.  Suppose $k > 1$.  As in the proof of
  Theorem~\ref{gencon}, we may assume that $f$ is nowhere continuous.

  Let $\mathcal{B}^H_0$ be a definable neighborhood basis of 0.  For
  $a \in X$ and $N \in \mathcal{B}^H_0$, let $a \bowtie N$ mean that
  $N$ strongly separates the elements of $f(a)$, in the sense that
  \begin{equation*}
    x,y \in f(a), x \ne y \implies x-N \cap y-N = \varnothing.
  \end{equation*}
  Because the topology on $H$ is Hausdorff, for any $a \in X$ there is
  some $N$ with $a \bowtie N$.  By the hammer lemma
  (Lemma~\ref{hammer}), there is a big subset $X' \subseteq X$ and an
  $N_0$ such that $a \bowtie N_0$ for any $a \in X'$.  Replacing $X'$
  with its interior, we may assume $X'$ is open in $G$.  Then $f
  \restriction X'$ is nowhere continuous.

  Let $\Gamma$ be the graph of $f \restriction X'$, i.e., $\{(x,y) \in
  X' \times H : y \in f(x)\}$.  Then $\Gamma$ has a finite-to-one
  projection onto $X'$, so $\dim(\Gamma) = \dim(X') = \dim(G) =: d$.
  Take a point $(a,b) \in \Gamma$ where the local dimension
  $\dim_{(a,b)}(\Gamma) = d$.  Then $G \times (b + N_0)$ is a
  neighborhood of $(a,b)$, so
  \begin{equation*}
    \dim((G \times (b + N_0)) \cap \Gamma) = d.
  \end{equation*}
  Let $X'' \subseteq X'$ be the projection of $(G \times (b + N_0))
  \cap Gamma$ onto the first coordinate.  As the projection is
  finite-to-one on $\Gamma$, we have $\dim(X'') = d$.  For any $x \in
  X''$, there is at least one $y$ in $f(x) \cap (b + N_0)$, by
  definition of $X''$.  If there was a second such $y'$, then
  \begin{equation*}
    y, y' \in b + N_0, \text{ so } b \in (y - N_0) \cap (y' - N_0)
  \end{equation*}
  contradicting the fact that $N_0$ strongly separates the elements of
  $f(x)$.  Therefore, for every $x \in X'$ there is a \emph{unique}
  point in $f(x) \cap (b + N_0)$.  Let $g(x)$ be this unique value.
  Then $g : X'' \to H$ is a definable function, and
  \begin{equation*}
    f(x) = \{g(x)\} \cup h(x) \text{ for } x \in X'',
  \end{equation*}
  for some definable $(k-1)$-correspondence $h : X'' \rightrightarrows
  H$.  By induction, $g$ and $h$ are continuous at almost every point
  of $X''$.  In particular, we can find a non-empty open set $U
  \subseteq X''$ such that $g$ and $h$ are continuous on $U$.  Then
  $f$ is continuous on $U$, a contradiction.
\end{proof}

\section{The case of fields} \label{elevensy}
\begin{theorem} \label{fieldtop}
  Let $(K,+,\cdot)$ be a definable field.  Let $\tau$ be the canonical
  topology on $(K,+)$.
  \begin{enumerate}
  \item $\tau$ is Hausdorff.
  \item $\tau$ is a field topology.
  \item $\tau$ restricted to $K^\times = K \setminus \{0\}$ agrees
    with the canonical topology of the multiplicative group.
  \end{enumerate}
\end{theorem}
\begin{proof}
  First note that for any $a \in K$, the multiplication-by-$a$ map
  $\mu_a : K \to K$ is continuous by Corollary~\ref{homcor}, or more
  simply by Remark~\ref{isorem}.  Since addition is continuous, any
  affine map $f(x) = ax+b$ is continuous, and any invertible affine
  map is a homeomorphism.
  \begin{enumerate}
  \item By Theorem~\ref{non-triv}, $\tau$ is non-trivial.  Therefore
    there is a neighborhood $U \ni 0$ and element $\delta$ with
    $\delta \notin U$.  Take $V \ni 0$ a smaller neighborhood such
    that $V-V \subseteq U$.  Given any distinct $a,b \in K$ let $s$ be
    such that $a + s\delta = b$.  Then $a+sV$ and $b+sV$ are
    neighborhoods of $a$ and $b$.  We claim that they are disjoint.
    Otherwise, there are $x,y \in V$ with
    \begin{gather*}
      a + sx = b + sy \\
      s\delta = b-a = s(x-y) \\
      \delta = x-y \in V-V \subseteq U,
    \end{gather*}
    a contradiction.
  \item By Theorem~\ref{product-top}, the canonical topology on
    $(K^2,+)$ is the product topology $(K,+) \times (K,+)$.  Then
    Theorem~\ref{gencon} gives \emph{some} point $(a,b)$ where the
    multiplication map $m(x,y) = xy$ is continuous.  Given any $c,d
    \in K$, the maps
    \begin{align*}
      (x,y) &\mapsto m(x-c,y-d) \\
      (x,y) &\mapsto m(x-c,y-d) + xd + yc - cd
    \end{align*}
    are continuous at $(a+c,b+d)$, because of the continuity of
    addition, $\mu_c$, and $\mu_d$.  But the second map is merely $m$
    because
    \begin{equation*}
      m(x-c,y-d) + xd + yc - cd = (x-c)(y-d) + xd + yc - cd = xy.
    \end{equation*}
    In particular, multiplication is continuous at $(a+c,b+d)$.  As
    $c$ and $d$ were arbitrary, multiplication is everywhere
    continuous.

    Similarly, let $i(x) = x^{-1}$.  Theorem~\ref{gencon} gives
    \emph{some} point $a$ where $i$ is continuous.  Using the strategy
    of Corollary~\ref{homcor}, one sees that $i$ is everywhere
    continuous.
  \item If we restrict $\tau$ to $K^\times$, the result is a definable
    group topology, by the previous point.  If $D \subseteq K^\times$
    is definable, then $D$ is big if and only if $D$ has non-empty
    $\tau$-interior, since this was already true in the bigger
    topological space $K$.  By Lemma~\ref{characterization}, $\tau
    \restriction K^\times$ is the canonical topology on $K^\times$.
    \qedhere
  \end{enumerate}
\end{proof}
\begin{corollary} \label{target}
  If $K$ is an infinite definable field, then the canonical topology
  on $K$ is st-henselian and $K$ is large.
\end{corollary}
\begin{proof}
  Assumptions~\ref{A} and \ref{B} hold:
  \begin{itemize}
  \item Most of the properties in Assumption~\ref{A} are formal
    properties of t-minimal dimension.
  \item $\dim(K) > 0$ because $K$ is infinite.
  \item The canonical topology is a field topology by
    Theorem~\ref{fieldtop}.
  \item The product topology on $K^n$ agrees with the canonical
    topology on $(K^n,+)$ (Theorem~\ref{product-top}), and so a subset
    $X \subseteq K^n$ has non-empty interior in the product topology
    if and only if it has maximum dimension (Theorem~\ref{quasi}).
  \item Definable correspondences are generically continuous by
    Theorem~\ref{gencon2}.  \qedhere
  \end{itemize}
\end{proof}
\begin{remark}
  If we just care about proving largeness, we only need
  bt-henselianity, so we only need Assumption~\ref{A}.  In this case,
  we can do without the generic continuity of correspondences, though
  we still need the generic continuity of functions to prove
  continuity of multiplication.
\end{remark}

\appendix

\section{Why we need open subsets of cells} \label{not-nice}
\begin{definition}
  A \emph{strict definable manifold} is a definable topological
  space $X$ covered by finitely many definable open subsets $U_i
  \subseteq X$, each of which is definably homeomorphic to an open
  subset of $\Mm^n$
\end{definition}
This would be analogous to the definition of ``definable manifold'' in
\cite{Peterzil-Steinhorn,admissible}.
\begin{proposition}
  In ACVF$_{0,0}$, there is a 1-dimensional definable group $G$ which
  is not a strict definable manifold with respect to any definable
  topology.
\end{proposition}
\begin{proof}
  Let $G$ be the elliptic curve $\{(x,y) \in \Mm^2 : y^2 = x^3 - x\}
  \cup \{\infty\}$, where $\infty$ is the point at infinity.  Then
  $\dim(G) = 1$.  If $G$ can be made into a strict manifold, then $G =
  \bigcup_{i=1}^n U_i$ for some non-empty definable sets $U_i$ with
  definable bijections $f_i : U_i \to V_i$ for some definable open
  sets $V_i \subseteq \Mm^{m_i}$.  Note that
  \begin{equation*}
    m_i = \dim(V_i) = \dim(U_i) \le \dim(G) = 1 \text{ for each $i$}.
  \end{equation*}
  Let $M \prec \Mm$ be a small model over which all the maps $f_i :
  U_i \to V_i$ are definable.  Let $a \in \Mm$ be such that $\tp(a/M)$
  is the generic type of the closed unit ball, meaning that $v(a) = 0$
  and $\res(a)$ is transcendental over $M$.  Let $b$ be one of the
  square roots of $a^3 - a$.  Then $(a,b) \in G$, so $(a,b)$ is in one
  of the sets $U_i$.  Then $f_i$ shows that $(a,b)$ is interdefinable
  over $M$ with some point $c \in V_i \subseteq
  \Mm^1$: \[\dcl(Mab) = \dcl(Mc).\] If $K$ is a subfield of $\Mm$ then
  $\dcl(K)$ is the henselization $K^h$, so
  \begin{equation*}
    M(a,b)^h = M(c)^h.
  \end{equation*}
  Taking residue fields of both sides, and using the fact that
  $\res(K^h) = \res(K)$, we see that
  \begin{equation*}
    \res(M(a,b)) = \res(M(c)). \tag{$\ast$}
  \end{equation*}
  Let $k$ be the residue field of $M$, and let $\alpha, \beta \in
  \res(\Mm)$ be the residues of $a$ and $b$, respectively.  Then
  \begin{equation*}
    k(\alpha,\beta) \subseteq \res(M(a,b)) = \res(M(c)) = k(\gamma)
  \end{equation*}
  for some $\gamma \in \res(\Mm)$, by Fact~\ref{somebody-knows} below.
  Because $\alpha$ is transcendental over $k$ and $\beta^2 = \alpha^3
  - \alpha$, the field $k(\alpha,\beta)$ is the function field $k(E)$,
  where $E$ is the elliptic curve $y^2 = x^3 - x$.  The inclusion
  $k(\alpha,\beta) \subseteq k(\gamma)$ contradicts L\"uroth's
  theorem, essentially.  More precisely, the inclusion
  $k(\alpha,\beta) \subseteq k(\gamma)$ corresponds to a dominant
  rational map $\Pp^1 \to E$ of $k$-varieties, but such a map cannot
  exist because $E$ has genus 1 and $\Pp^1$ has genus 0.
\end{proof}
\begin{fact} \label{somebody-knows}
  Let $\Mm$ be a monster model of ACVF.  Let $M \prec \Mm$ be a small
  submodel and let $k$ be the residue field $\res(M) \subseteq
  \res(\Mm)$.  If $a \in \Mm^1$, then the residue field of $M(a)$ is
  either $k$ or $k(\alpha)$ for some $\alpha \in \res(\Mm)$.
\end{fact}
Fact~\ref{somebody-knows} is well-known, and comes from the
classification of 1-types in ACVF.

\section{The lemma on products} \label{app-B}
In this appendix, we prove Fact~\ref{prod-fact}.
\subsection{Dimension independence}
\begin{remark} 
  If $a_1, a_2$ are finite tuples and $B$ is a set of parameters, then
  \begin{equation*}
    \dim(a_1a_2/B) \le \dim(a_1/a_2B) + \dim(a_2/B) \le \dim(a_1/B) +
    \dim(a_2/B),
  \end{equation*}
  by subadditivity of dimension
  \cite[Proposition~2.31(5)]{wj-visc-1}.  More generally,
  \begin{equation*}
    \dim(a_1a_2 \cdots a_n/B) \le \sum_{i=1}^n \dim(a_i/B)
  \end{equation*}
  by induction on $n$.
\end{remark}
\begin{definition} 
  A finite sequence of tuples $(a_i : i \in I)$ is \emph{dimension
  independent} over $B$ if
  \begin{equation*}
    \dim(\ba/B) = \sum_{i \in I} \dim(a_i/B),
  \end{equation*}
  where $\ba$ is the concatenation of the tuples.
\end{definition}
\begin{lemma} \label{subtuple}
  If a finite sequence $(a_i : i \in I)$ is dimension independent over
  $B$ and if $I_0 \subseteq I$, then the subsequence $(a_i : i \in
  I_0)$ is dimension independent over $B$.
\end{lemma}
\begin{proof}
  Without loss of generality, $a_1,\ldots,a_n$ is dimension
  independent over $B$ and we must show that $a_1,\ldots,a_m$ is
  dimension independent, for some $m < n$.  Otherwise,
  \begin{gather*}
    \dim(a_1,\ldots,a_m/B) < \sum_{i=1}^m \dim(a_i/B) \\
    \dim(a_{m+1},\ldots,a_n/B) \le \sum_{i=m+1}^n \dim(a_i/B),
  \end{gather*}
  and so
  \begin{gather*}
    \dim(a_1,\ldots,a_n/B) \le \dim(a_1,\ldots,a_m/B) +
    \dim(a_{m+1},\ldots,a_n/B) < \sum_{i=1}^n \dim(a_i/B),
  \end{gather*}
  a contradiction.
\end{proof}
Then the following definition makes sense.
\begin{definition}
  An infinite
  sequence $(a_i : i \in I)$ is \emph{dimension independent} over $B$
  if any finite subsequence is.
\end{definition}
\begin{remark} \label{broad-indep}
  Let $a_1,\ldots,a_n$ be a sequence of tuples such that $\tp(a_i/B)$
  is broad for each $i$.  Then $\dim(a_i/B)$ is the length $|a_i|$,
  and so the following are equivalent:
  \begin{itemize}
  \item $a_1,\ldots,a_n$ is dimension independent over $B$.
  \item $\dim(a_1,\ldots,a_n/B) = \sum_{i=1}^n |a_i|$.
  \item $\dim(a_1,\ldots,a_n/B)$ equals the length of the tuple
    $(a_1,\ldots,a_n)$.
  \item $\tp(a_1,\ldots,a_n/B)$ is broad.
  \end{itemize}
\end{remark}
\begin{remark} \label{move-by-acl}
  Suppose $a_1,\ldots,a_n$ is dimension independent over $C$, and
  $a'_1,\ldots,a'_n$ is another sequence such that $\acl(Ca_i) =
  \acl(Ca'_i)$ for each $i$.  Then $a'_1,\ldots,a'_n$ is dimension
  independent over $C$.  Indeed,
  \begin{gather*}
    \dim(a'_i/C) = \dim(a_i/C) \text{ for each $i$} \\
    \dim(a'_1,\ldots,a'_n/C) = \dim(a_1,\ldots,a_n/C)
  \end{gather*}
  by \cite[Proposition~2.31(1)]{wj-visc-1}.
\end{remark}

\subsection{The cloning lemma}

\begin{lemma}
  The following are equivalent for an $n$-tuple $a \in \Mm^n$ and a
  small set of parameters $B$:
  \begin{enumerate}
  \item $\tp(a/B)$ is broad.
  \item There is a mutually $B$-indiscernible array $\{c_{i,j}\}_{1
    \le i \le n, ~ 0 \le j < \omega}$ whose first column
    $\{c_{i,0}\}_{1 \le i \le n}$ equals $a$, and the elements in each
    row are distinct.
  \end{enumerate}
\end{lemma}
\begin{proof}
  We claim that (1) and (2) are equivalent to (3) and (4):
  \begin{enumerate}
    \setcounter{enumi}{2}
  \item There is a mutually $B$-indiscernible array $\{c_{i,j}\}_{1
    \le i \le n, ~ 0 \le j < \omega}$ such that for every function
    $\eta : \{1,\ldots,n\} \to \omega$, the tuple
    $\{c_{i,\eta(i)}\}_{1 \le i \le n}$ realizes $\tp(a/B)$, and in
    each row the elements are distinct.
  \item There is an array $\{c_{i,j}\}_{1 \le i \le n, ~ 0 \le j <
    \omega}$ such that for every function $\eta : \{1,\ldots,n\} \to
    \omega$, the tuple $\{c_{i,\eta(i)}\}_{1 \le i \le n}$ realizes
    $\tp(a/B)$, and in each row the elements are distinct.
  \end{enumerate}
  The equivalence of (1) and (4) is essentially the definition of
  ``broad'', and the equivalence of (3) and (4) holds because we can
  extract mutually indiscernible arrays.  Condition (2) implies
  condition (3) because the array from (2) must satisfy the condition
  in (3) by mutual indiscernibility.  Finally, (3) implies (2) because
  given an array as in (3), we can move it by an automorphism over $B$
  to make the first column equal $a$.
\end{proof}
Say that an array $\{c_{i,j}\}$ \emph{witnesses broadness} of
$\tp(a/B)$ if it is mutually indiscernible with first column equal to
$a$.

The following is well-known:
\begin{fact} \label{acl-irrel}
  A sequence is indiscernible over a set $A$ iff it
  is indiscernible over $\acl(A)$.
\end{fact}
Indeed, if $I$ is $A$-indiscernible and $J$ is an
$\acl(A)$-indiscernible sequence of the same length, extracted from
$I$, then $I \equiv_A J$.  Taking $\sigma \in \Aut(\Mm/A)$ with
$\sigma(J)=I$, we see that $I=\sigma(J)$ is indiscernible over
$\sigma(\acl(A)) = \acl(A)$.
\begin{lemma}[Cloning lemma] \label{cloning}
  Let $a,b$ be two finite tuples which are dimension independent over
  $C$.  Then there is $b'$ such that $a,b,b'$ are dimension
  independent over $C$ and $ab' \equiv_C ab$.
\end{lemma}
\begin{proof}
  To simplify notation, suppose $\dim(a/C) = 3$ and $\dim(b/C) = 2$.
  Take $\alpha \in \Mm^3$ and $\beta \in \Mm^2$ to be acl-bases of $a$
  and $b$ over $C$.  Then $\alpha, \beta$ are dimension-independent
  over $C$ (Remark~\ref{move-by-acl}), so $\tp(\alpha,\beta/C)$ is
  broad (Remark~\ref{broad-indep}).  Take a $5 \times \omega$ mutually
  $C$-indiscernible array witnessing broadness of
  $\tp(\alpha\beta/C)$.  In block notation, this array looks like
  \begin{equation*}
    \begin{pmatrix}
      A \\
      B
    \end{pmatrix}
  \end{equation*}
  where $A$ is a $3 \times \omega$ array whose first column is
  $\alpha$, and $B$ is a $2 \times \omega$ array whose first column is
  $\beta$.  Extend the array on the right, to a $5 \times (\omega +
  \omega)$ mutually $C$-indiscernibe array
  \begin{equation*}
    \begin{pmatrix}
      A & A' \\
      B & B'
    \end{pmatrix}
  \end{equation*}
  where $A'$ is a $3 \times \omega$ array and $B'$ is a $2 \times
  \omega$ array.  Let $\beta'$ be the first column of $B'$.  Then
  \begin{equation*}
    \begin{pmatrix}
      A \\
      B \\
      B'
    \end{pmatrix}
  \end{equation*}
  is a mutually $C$-indiscernible array witnessing that
  $\tp(\alpha,\beta,\beta'/C)$ is broad.  In particular, the three
  tuples $\alpha, \beta, \beta'$ are dimension-independent over $C$ (Remark~\ref{broad-indep}).
  Note that the array
  \begin{equation*}
    \begin{pmatrix}
      B & B'
    \end{pmatrix}
  \end{equation*}
  is mutually indiscernible over $CA$, and in particular over
  $C\alpha$.  So its sequence of columns is indiscernible over
  $C\alpha$, hence indiscernible over $\acl(C\alpha) \supseteq Ca$ by
  Fact~\ref{acl-irrel}.  Therefore, $\beta \equiv_{Ca} \beta'$.  Take
  $b'$ such that $\beta b \equiv_{Ca} \beta' b'$.  Then $b'$ is
  interalgebraic with $\beta'$.  The fact that $\alpha, \beta, \beta'$
  is $C$-independent implies that $a, b, b'$ is $C$-independent
  (Remark~\ref{move-by-acl}).  We arranged $b \equiv_{Ca} b'$.
\end{proof}

\subsection{The lemma on multiplicities}
\begin{remark} \label{help1}
  If $a_1,a_2,a_3$ are dimension independent over $B$, then $a_1a_2$ and $a_3$ are dimension independent over $B$.  Otherwise,
  \begin{equation*}
    \dim(a_1a_2a_3/B) < \dim(a_1a_2/B) + \dim(a_3/B) \le \dim(a_1/B) + \dim(a_2/B) + \dim(a_3/B).
  \end{equation*}
\end{remark}
\begin{remark} \label{help2}
  If $a_1, a_2$ are dimension independent over $B$, then
  $\dim(a_1/Ba_2) = \dim(a_1/B)$.  Otherwise, $\dim(a_1/Ba_2) <
  \dim(a_1/B)$ and so
  \begin{equation*}
    \dim(a_1a_2/B) \le \dim(a_1/Ba_2) + \dim(a_2/B) < \dim(a_1/B) +
    \dim(a_2/B)
  \end{equation*}
  by subadditivity of dimension
  \cite[Proposition~2.31(5)]{wj-visc-1}.
\end{remark}
If $a \in \acl(B)$, let $\mult(a/B)$ denote the number of conjugates
of $a$ over $B$, as in Section~\ref{somewhere}.  Note that if $a \in
\acl(B)$ and $B \subseteq B'$, then $\mult(a/B') \le \mult(a/B)$ since
$\tp(a/B')$ has fewer realizations than $\tp(a/B)$.  By the orbit
stabilizer theorem, $\mult(a/B)$ is the index of $\Aut(\Mm/aB)$ in
$\Aut(\Mm/B)$.
\begin{lemma}[Lemma on multiplicities] \label{mult-lem}
  Let $a$ and $b$ be dimension independent over $C$.  Let $a_0$ and
  $b_0$ be $\acl$-bases of $a$ and $b$ over $C$, respectively.  Then
  there is a small set $C' \supseteq C$ such that the following hold:
  \begin{itemize}
  \item $a$ and $b$ are dimension independent over $C'$.
  \item $\dim(a/C') = \dim(a/C)$ and $\dim(b/C') = \dim(b/C)$.
  \item $a_0$ and $b_0$ are $\acl$-bases of $a$ and $b$ over $C'$.
  \item $\mult(ab/C'a_0b_0) = \mult(a/C'a_0)\mult(b/C'b_0)$.
  \end{itemize}
\end{lemma}
\begin{proof}
  Let $\mathcal{F}$ be the family of sets $C' \subseteq \Mm$ such that
  $|C'| \le |C| + \aleph_0$ and $C' \supseteq C$ and $\dim(a,b/C') =
  \dim(a,b/C)$.  If $C'$ is in $\mathcal{F}$, then
  \begin{equation*}
    \dim(a,b/C) = \dim(a,b/C') \le \dim(a/C') + \dim(b/C') \le \dim(a/C)
    + \dim(b/C),
  \end{equation*}
  so the inequalities are equalities:
  \begin{gather*}
    \dim(a,b/C') = \dim(a/C') + \dim(b/C') \\
    \dim(a/C') = \dim(a/C) \text{ and } \dim(b/C') = \dim(b/C).
  \end{gather*}
  In particular, $a$ and $b$ are dimension independent over any $C'
  \in \mathcal{F}$.  Since $C' \supseteq C$, it continues to be true
  that $a_0$ and $a$ are interalgebraic over $C'$.
  Then \[\dim(a_0/C') = \dim(a/C') = \dim(a/C) = \dim(a_0/C) =
  |a_0|,\] so $a_0$ is acl-independent over $C'$, and $a_0$ is an
  acl-basis of $a$.  Similarly, $b_0$ is an acl-basis of $b$.

  Take $C' \in \mathcal{F}$ minimizing $\mult(a/C'a_0) +
  \mult(b/C'b_0)$.  For any larger $C'' \in \mathcal{F}$, we have
  \begin{equation*}
    \mult(a/C''a_0) + \mult(b/C''b_0) \le \mult(a/C'a_0) + \mult(a/C'b_0),
  \end{equation*}
  so equalities hold:
  \begin{gather*}
    \mult(a/C''a_0) = \mult(a/C'a_0) \tag{$\ast$} \\
    \mult(b/C''b_0) = \mult(b/C'b_0).
  \end{gather*}
  As noted above, $a$ and $b$ are dimension independent over $C'$.  By
  the cloning lemma (Lemma~\ref{cloning}), there is $b'$ such that
  $a,b,b'$ are dimension independent over $C'$, and $ab \equiv_{C'}
  ab'$.  By Remarks~\ref{help1} and \ref{help2}, $ab$ is dimension
  independent from $b'$ and $\dim(ab/C'b') = \dim(ab/C') =
  \dim(ab/C)$.  Then $C'b' \in \mathcal{F}$.  By ($\ast$), we have
  \begin{gather*}
    \mult(a/C'a_0) = \mult(a/C'b'a_0) = \mult(a/C'ba_0)
  \end{gather*}
  since $ab \equiv_{C'} ab'$.  From the
  inequality $\mult(a/C'a_0) \ge \mult(a/C'b_0a_0) \ge
  \mult(a/C'ba_0)$ we see
  \begin{equation*}
    \mult(a/C'a_0) = \mult(a/C'a_0b_0) = \mult(a/C'a_0b).
  \end{equation*}
  By symmetry, we also get
  \begin{equation*}
    \mult(b/C'b_0) = \mult(b/C'a_0b_0) = \mult(b/C'ab_0).
  \end{equation*}
  Finally, the multiplicity $\mult(ab/C'a_0b_0)$ can be calculated as
  the index of certain automorphism groups:
  \begin{align*}
    \mult(ab/C'a_0b_0) &= |\Aut(\Mm/C'a_0b_0) : \Aut(\Mm/C'ab)| \\
    &= |\Aut(\Mm/C'a_0b_0) : \Aut(\Mm/C'ab_0)| \cdot |\Aut(\Mm/C'ab_0) : \Aut(\Mm/C'ab)| \\
    &= \mult(a/C'a_0b_0) \cdot \mult(b/C'ab_0) \\
    &= \mult(a/C'a_0) \cdot \mult(b/C'b_0).  \qedhere
  \end{align*}
\end{proof}

\subsection{The lemma on products}
\begin{lemma} \label{open-lemma}
  If $p \in S_n(B)$ is broad, then the set of realizations $p(\Mm) =
  \{a \in \Mm^n : a \models p\}$ is open.
\end{lemma}
\begin{proof}
  Fix $a$ in $p(\Mm)$.  We claim that there is a basic open set $B$
  such that $a \in B$ and $B \subseteq \phi(\Mm)$ for every formula
  $\phi \in p$.  Otherwise, by saturation, there are finitely many
  formulas $\phi_1,\ldots,\phi_n \in p$ such that \emph{no} basic open
  set $B$ satisfies $a \in B \subseteq \bigcap_{i=1}^n \phi_i(\Mm)$.
  Equivalently, $a$ is not in the interior of the $B$-definable set $D
  = \bigcap_{i=1}^n \phi_i(\Mm)$.  However, $a \in D$ because $a
  \models p$ and each $\phi_i$ is from $p$.  Thus
  \begin{equation*}
    a \in D \setminus \ter(D).
  \end{equation*}
  The set on the right is $B$-definable, with empty interior, so it is
  narrow \cite[Proposition~2.7]{wj-visc-1} contradicting the fact
  that $p = \tp(a/B)$ is broad.
\end{proof}
%% Recall the notion of ``weak $k$-cells'' from
%% \cite[Definition~2.47]{wj-visc-1} or Section~\ref{somewhere}.
\begin{lemma}[Lemma on products]
  Suppose $X$ and $Y$ are non-empty definable sets and $D \subseteq X
  \times Y$ with $\dim(D) = \dim(X) + \dim(Y)$.  Then there are
  definable sets $X_0 \subseteq X$ and $Y_0 \subseteq Y$ with
  \begin{gather*}
    X_0 \times Y_0 \subseteq D \\
    \dim(X_0) = \dim(X) \text{ and } \dim(Y_0) = \dim(Y).
  \end{gather*}
\end{lemma}
\begin{proof}
%%   By \cite[Lemma~2.49]{wj-visc-1}, we can write $X$ and $Y$ as
%%   disjoint unions of weak cells:
%%   \begin{gather*}
%%     X = \coprod_{i=1}^n X_i \\
%%     Y = \coprod_{j=1}^m Y_j.
%%   \end{gather*}
%%   One of the intersections $D \cap (X_i \times Y_j)$ must have the
%%   same dimension as $X \times Y$.  Then
%%   \begin{gather*}
%%     \dim(D \cap (X_i \times Y_j)) = \dim(X_i \times Y_j) \\
%%     \dim(X_i) = \dim(X) \text{ and } \dim(Y_j) = \dim(Y).
%%   \end{gather*}
%%   Replacing $X$ with $X_i$, $Y$ with $Y_j$, and $D$ with $D \cap (X_i
%%   \times Y_j)$, we reduce to the case where $X$ and $Y$ are weak
%%   cells.  Let $\pi_1$ and $\pi_2$ be the coordinate projections
%%   witnessing that $X$ and $Y$ are weak cells.  Thus $X \to \pi_1(X)$
%%   is finite-to-one map from $X$ onto an open set $\pi_1(X)$, and
%%   similarly for $Y$ and $\pi_2$.
  
  Take a small set $C$ defining the sets $X,Y,D$.  Take $(a,b) \in D$
  with $\dim(a,b/C) = \dim(D) = \dim(X) + \dim(Y)$.  Then
  \begin{equation*}
    \dim(a,b/C) \le \dim(a/C) + \dim(b/C) \le \dim(X) + \dim(Y)
  \end{equation*}
  so the inequalities are equalities, implying that $a$ and $b$ are
  dimension independent over $C$ and
  \begin{gather*}
    \dim(a/C) = \dim(X) \\
    \dim(b/C) = \dim(Y).
  \end{gather*}
  Let $a_0$ and $b_0$ be
%%   Let $a_0 = \pi_1(a)$ and $b_0 = \pi_2(b)$.  Then $a_0,b_0$ are
  acl-bases of $a$ and $b$ over $C$.  % by Remark~\ref{weak-cell-basis}.
  By the lemma on multiplicities (Lemma~\ref{mult-lem}), we can
  replace $C$ with a bigger set and arrange for
  \begin{equation*}
    \mult(ab/Ca_0b_0) = \mult(a/Ca_0)\mult(b/Cb_0). \tag{$\ast$}
  \end{equation*}
  Let $\phi$ be an $\Ll(C)$-formula such that $\phi(\Mm,a_0)$ is the
  set of conjugates of $a$ over $Ca_0$.  Thus $|\phi(\Mm,a_0)| =
  \mult(a/Ca_0)$.  Similarly, let $\psi$ be an $\Ll(C)$-formula such
  that $\psi(\Mm,b_0)$ is the set of conjugates of $b$ over $Cb_0$.
  Then $|\psi(\Mm,b_0)| = \mult(b/Cb_0)$.
  \begin{claim} \label{whatever}
    If $a' \in \phi(\Mm,a_0)$ and $b' \in \psi(\Mm,b_0)$, then $a'b'
    \equiv_{Ca_0b_0} ab$.
  \end{claim}
  \begin{claimproof}
    Let $S$ be the set of conjugates of $ab$ over $Ca_0b_0$.  Clearly
    $S \subseteq \phi(\Mm,a_0) \times \psi(\Mm,b_0)$.  But ($\ast$)
    shows
    \begin{equation*}
      |S| = \mult(ab/Ca_0b_0) = \mult(a/Ca_0)\mult(b/Cb_0) =
      |\phi(\Mm,a_0)| \cdot |\psi(\Mm,b_0)|.
    \end{equation*}
    Thus $S = \phi(\Mm,a_0) \times \psi(\Mm,b_0)$.
  \end{claimproof}
  Note that $a_0$ and $b_0$ are dimension independent over $C$ because
  they are interalgebraic with $a$ and $b$ (see
  Remark~\ref{move-by-acl}), and therefore $\tp(a_0,b_0/C)$ is broad
  (see Remark~\ref{broad-indep}).  By Lemma~\ref{open-lemma}, the set
  of realizations of $\tp(a_0,b_0/C)$ is open.  Then it contains a
  product $U \times V$ where $U$ is a definable open neighborhood of
  $a_0$ and $V$ is a definable open neighborhood of $b_0$.  (The sets
  $U$ and $V$ need not be $C$-definable.)  Let $\pi_1$ and $\pi_2$ be
  the coordinate projections such that $\pi_1(a) = a_0$ and $\pi_2(b)
  = b_0$.  Let
  \begin{gather*}
    X_0 = \{a' \in X : \pi_1(a') \in U \text{ and } a' \in \phi(\Mm,\pi_1(a'))\} \\
    Y_0 = \{b' \in Y : \pi_2(b') \in V \text{ and } b' \in \psi(\Mm,\pi_2(b'))\}.
  \end{gather*}
  \begin{claim}
    The map $\pi_1 : X_0 \to U$ is surjective.
  \end{claim}
  \begin{claimproof}
    Given $\alpha \in U$, take any $\beta \in V$.  Then
    $(\alpha,\beta) \in U \times V$, so $\alpha\beta \equiv_C a_0b_0$.
    Take $a'$ such that \[a'\alpha\beta \equiv_C aa_0b_0.\] The fact
    that $\pi_1(a) = a_0$ implies that $\pi_1(a') = \alpha$.  The fact
    that $a \in X$ implies $a' \in X$ (since $X$ is $C$-definable).
    The fact that $a \in \phi(\Mm,a_0)$ implies $a' \in
    \phi(\Mm,\alpha)$ (since $\phi$ is a formula over $C$).  Then $a'
    \in X_0$.
  \end{claimproof}
  \begin{claim}
    The map $\pi_1 : X_0 \to U$ has finite fibers.
  \end{claim}
  \begin{proof}
    If $\alpha \in U$, then as in the proof of the previous claim,
    $\alpha \equiv_C a_0$, and then $\phi(\Mm,\alpha)$ is finite
    because $\phi(\Mm,a_0)$ is.  The fiber of $\pi_1 : X_0 \to U$ over
    $\alpha$ is the set $\{a' \in X : a' \in \phi(\Mm,\alpha)\}$,
    which is finite.
  \end{proof}
  It follows that \[\dim(X_0) \ge \dim(U) = |a_0| = \dim(a/C) =
  \dim(X).\] Therefore $\dim(X_0) = \dim(X)$, and similarly $\dim(Y_0)
  = \dim(Y)$.

  It remains to show that $X_0 \times Y_0 \subseteq D$.  Take $(a',b')
  \in X_0 \times Y_0$.  Let $a'_0 = \pi_1(a')$ and $b'_0 = \pi_2(b')$.
  By definition of $X_0$ and $Y_0$, we have
  \begin{gather*}
    (a'_0,b'_0) \in U \times V \\
    a' \in \phi(\Mm,a'_0) \text{ and } b' \in \psi(\Mm,b'_0).
  \end{gather*}
  By choice of $U \times V$, $a'_0b'_0 \equiv_C a_0b_0$.  Take $\sigma
  \in \Aut(\Mm/C)$ with $\sigma(a'_0b'_0) = a_0b_0$.  Then
  \begin{gather*}
    \sigma(a') \in \phi(\Mm,\sigma(a'_0)) = \phi(\Mm,a_0) \\
    \sigma(b') \in \psi(\Mm,\sigma(b'_0)) = \psi(\Mm,b_0).
  \end{gather*}
  By Claim~\ref{whatever}, $a'b' \equiv_C \sigma(a'b') \equiv_C ab$
  and so $(a',b') \in D$ because $(a,b) \in D$.  This proves that $X_0
  \times Y_0 \subseteq D$.
\end{proof}

\begin{acknowledgment}
  The author was supported by the National Natural Science Foundation
  of China (Grants No.\@ 12101131 and W2532009) as well as the Ministry of Education
  of China (Grant No.\@ 22JJD110002).
\end{acknowledgment}

\bibliographystyle{plain} \bibliography{mybib}{}

\begin{thebibliography}{10}

\bibitem{acosta-hasson}
Juan~Pablo Acosta~L{\'{o}}pez and Assaf Hasson.
\newblock On groups and fields definable in 1-h-minimal fields.
\newblock {\em J. Inst. Math. Jussieu}, 24(1):203--248, 2025.

\bibitem{castle-hasson}
Benjamin Castle and Assaf Hasson.
\newblock Topologically 1-based t-minimal structures.
\newblock {arXiv:2508.18558v1 [math.LO]}, August 2025.

\bibitem{hensquot2}
Philip Dittmann, Erik Walsberg, and Jinhe Ye.
\newblock When is the {\'e}tale open topology a field topology?
\newblock {\em Israel Journal of Mathematics}, 2025.

\bibitem{visceral}
Alfred Dolich and John Goodrick.
\newblock Tame topology over definable uniform structures.
\newblock {\em Notre Dame J. Formal Logic}, 63(1):51--79, 2022.

\bibitem{admissible}
Will Johnson.
\newblock Topologizing interpretable groups in $p$-adically closed fields.
\newblock {\em Notre Dame J. Formal Logic}, 64(4):571--609, November 2023.

\bibitem{wj-visc-1}
Will Johnson.
\newblock Visceral theories without assumptions.
\newblock {arXiv:2404.11453v1 [math.LO]}, April 2024.

\bibitem{large-gt}
Will Johnson.
\newblock Largeness and generalized t-henselianity.
\newblock {arXiv:2508.15362v1 [math.LO]}, August 2025.

\bibitem{tops-rings}
Will Johnson.
\newblock Translating between {NIP} integral domains and topological fields.
\newblock {\em J. Math. Logic}, page Online ready, 2026.

\bibitem{CXF}
Will Johnson and Jinhe Ye.
\newblock Curve-excluding fields.
\newblock {\em J. Eur. Math. Soc.}, page Online first, 2025.

\bibitem{mathews}
Larry Mathews.
\newblock Cell decomposition and dimension functions in first-order topological
  structures.
\newblock {\em Proc. London Math. Soc.}, 70(3):1--32, 1995.

\bibitem{red-book}
David Mumford.
\newblock {\em The Red Book of Varieties and Schemes}, volume 1358 of {\em
  Lecture Notes in Mathematics}.
\newblock Springer Berlin / Heidelberg, 2004.

\bibitem{Peterzil-Steinhorn}
Y.~Peterzil and C.~Steinhorn.
\newblock Definable compactness and definable subgroups of o-minimal groups.
\newblock {\em Journal of the London Mathematical Society}, 59(3):769--786,
  1999.

\bibitem{Pop-little}
Florian Pop.
\newblock Little survey on large fields - old {\&} new.
\newblock In {\em Valuation Theory in Interaction}, pages 432--463. European
  Mathematical Society Publishing House, 2014.

\bibitem{prestel-ziegler}
Alexander Prestel and Martin Ziegler.
\newblock Model theoretic methods in the theory of topological fields.
\newblock {\em Journal f\"ur die reine und angewandte Mathematik},
  299-300:318--341, 1978.

\bibitem{simonWalsberg}
Pierre Simon and Erik Walsberg.
\newblock Tame topology over dp-minimal structures.
\newblock {\em Notre Dame J. Formal Logic}, 60(1):61--76, 2019.

\end{thebibliography}
\end{document}